\documentclass{amsart}
\usepackage[utf8x]{inputenc}
\usepackage{amsmath,amsthm,amsfonts,amstext,amssymb}
\usepackage{mathtools}
\usepackage[margin=1in]{geometry}
\usepackage{graphicx,enumitem}

\usepackage[T1]{fontenc}
\usepackage{scalefnt}
\RequirePackage[numbers]{natbib}
\RequirePackage[colorlinks,citecolor=blue,urlcolor=blue]{hyperref}

\usepackage{tikz,pgfplots,verbatim}
\usetikzlibrary{calc}

\newtheorem{theorem}{Theorem}
\newtheorem{prop}[theorem]{Proposition}
\newtheorem{lemma}[theorem]{Lemma}
\theoremstyle{remark}

\newtheorem*{example}{Example}

\newcommand{\levy}{L\'{e}vy }
\newcommand{\p}{{\mathbb P}}
\newcommand{\e}{{\mathbb E}}
\newcommand{\Q}{{\mathbb Q}}
\newcommand{\D}{{\mathrm d}}

\newcommand{\N}{{\mathbb N}}
\newcommand{\R}{{\mathbb R}}
\newcommand{\1}[1]{\mbox{\rm\large  1}_{\{#1\}}}

\renewcommand{\a}{{\alpha}}

\newcommand{\ov}[1]{{\overline{#1}}}

\newcommand{\n}{{(n)}}

\newcommand{\cip}{ \,{\stackrel{\p}{\to}}\, }

\newcommand{\Oh}{{\mathcal O}}
\newcommand{\sign}{{\rm sign}}

\newcommand{\convd}{{\stackrel{\mathrm d}{\,\to\,}}}

\begin{document}
\author[J.\ Gonz\'alez C\'azares and J.\ Ivanovs]{Jorge Gonz\'alez C\'azares and Jevgenijs Ivanovs}
\address{University of Warwick and Aarhus University}

\title[Recovering Brownian and jump parts]{Recovering Brownian and jump parts from high-frequency observations of a L\'evy process}

\begin{abstract}
We introduce two general non-parametric methods for recovering paths of the 
Brownian and jump components from high-frequency observations of a L\'evy 
process. The first procedure relies on reordering of independently sampled 
normal increments and thus avoids tuning parameters. The functionality of this 
method is a consequence of the small time predominance of the Brownian 
component, the presence of exchangeable structures, and fast convergence of 
normal empirical quantile functions. 
The second procedure amounts to filtering the increments and 
compensating with the final value. It requires a carefully chosen threshold, 
in which case both methods yield the same rate of convergence. This rate 
depends on the small-jump activity and is given in terms of the 
Blumenthal-Getoor index. Finally, we discuss possible extensions, including 
the multidimensional case, and provide numerical illustrations. 
\end{abstract}	

\maketitle

%
%
%

\section{Introduction}
Consider a L\'evy process $X$ on $[0,1]$ and the decomposition 
\begin{equation}\label{eq:dec}
X_t=Y_t+\sigma W_t,\qquad t\in[0,1],
\end{equation}
where $\sigma\geq 0$ and $W$ is a standard Brownian motion independent of 
the L\'evy process $Y$ with no Brownian component. In this work we assume 
that $\sigma>0$ and provide two methods to recover $W$, and thus also $Y$, 
from high-frequency observations $(X_{i/n})_{i=0,\ldots,n}$ (as $n\to\infty$) 
of a given sample path of~$X$. 

More precisely, we recover the path of the bridge $(W_t-W_1t)_{t\in[0,1]}$ and 
the path of the drifted process $(Y_t+\sigma W_1t)_{t\in[0,1]}$. It is not possible 
to recover $W_1$ because it is impossible to separate a linear drift from the 
Brownian path consistently (as the respective laws are equivalent). 
Note, however, that if $Y$ has bounded variation on compacts, there is a clear 
definition of the linear drift. Using the second method, this drift can be naturally 
separated from $Y$ (but not from $W$). Importantly, the proposed procedures 
do not require knowledge on the law of $X$, except for the parameter~$\sigma$ 
(to some extent), which can be estimated efficiently from the given 
high-frequency observations~\cite{jacod_survey,jacod_reiss}. Furthermore, 
the first method, which is the main focus of this paper, completely avoids 
tuning parameters. 

Apart from their intrinsic interest, the discussed  procedures may be useful in a 
variety of applied areas. Oftentimes $\sigma W$ is interpreted as noise, 
see e.g.~\cite{NoiseReview,LevyNoise,MR1870962,PhysicsNoise,NeuralBM}, 
and thus our methods recovers the signal $Y$ up to an unknown linear drift. 
This separation can then be used to answer various further questions regarding 
the observed trajectory and its decomposition. For example, what is the 
maximal fluctuation of the signal around its linear drift 
$\sup_{t\in[0,1]}|Y_t-tY_1|$ and how does it compare with the same quantity 
for the noise component?

Various statistical procedures may benefit from pre-separation of 
the Brownian part. According to~\cite[\S 5]{masuda} `coexistence of the Gaussian 
part and the jump part makes the parametric estimation problem much more 
difficult and cumbersome' and the common strategy then is to use thresholding. 
As was exemplified in~\cite{MR2510946} through simulations, a naive choice of 
the threshold may severely deteriorate estimation performance (see also 
Proposition~\ref{prop:neg} below). Thus our first procedure can be employed to 
avoid the difficult practical problem of threshold selection. 
Furthermore, it can be used as an alternative to~\cite{Lee_jumps2,Lee_jumps1} 
to detect the presence of jumps. It must be noted, however, that essentially 
simpler problems than path decomposition may suffer from suboptimal rates 
coming from the latter part. Nevertheless, absence of tuning 
parameters may still seem attractive in applications.
 
Finally, our first method provides a coupling between the L\'evy process 
$X$ and a Brownian motion by means of approximating $W$ which can be of independent interest, 
see~\cite{blanchet} for an application where upper bounds on the 
Wasserstein distance between the laws of a L\'evy process and a Brownian 
motion are needed. This coupling is easy to simulate, making it suitable 
for (multilevel) Monte Carlo methods. We note here that, for any fixed $n$, 
the approximation of $W$ produced by the second method need not be a 
(discretely observed) Brownian motion.

\subsection{Method I: reordering of normal increments}\label{sec:meth1}
Our main procedure, which may be surprising at first, 
has a simple construction:
\begin{itemize}
	\item[a)] simulate an independent standard Brownian motion $W'$ on the grid 
	$(i/n)_{i=1,\ldots, n}$,
	\item[b)] reorder the increments of $W'$ according to the ordering of the increments of $X$.
\end{itemize}
We will show that the resultant skeleton $W^\n$ 
(see~\eqref{eq:Wprime} below for definition) satisfies: 
\begin{equation}\label{eq:meth1}
\big(W_t-W_t^\n\big)_{t\in[0,1]}\,\cip\, \big(\big(W_1-W'_1\big)t\big)_{t\in[0,1]}
\end{equation}
in supremum norm. In words, we recover the Brownian evolution up to some 
linear drift. In fact, we have a much stronger result in Theorem~\ref{thm:gral_rate} 
establishing the speed of convergence, see also Figure~\ref{fig:log_table} for a 
numerical illustration. The joint recovery is now straightforward:
\begin{equation}\label{eq:joint}
\Big(W^\n_t-W^\n_1t,X_t-\sigma(W^\n_t-W^\n_1t)\Big)
	\cip \Big(W_t-W_1 t,Y_t+\sigma W_1t\Big)
\end{equation}
in supremum norm, which  is assumed throughout unless mentioned otherwise.
Note, however, that only $X_{\lfloor tn\rfloor/n}$ are available but we may 
replace $X_t$ above by its discretised version while relaxing to 
the convergence in Skorokhod~$J_1$-topology~\cite[A2]{MR1876169}. Alternatively, we may 
look at the difference of both sides and discretise all the processes involved.

Figure~\ref{fig:X} illustrates the algorithm in the case $\sigma=1$ and $Y$ being 
a variance gamma process. In addition, we remove the random drift in the 
approximation $X-W^\n$ of $Y$ by matching the endpoints. In general, 
this is not possible in practice, and is done here only to assess the signal recovery. 

\begin{figure}[h!]
\begin{center}
\includegraphics[width=0.99\textwidth]{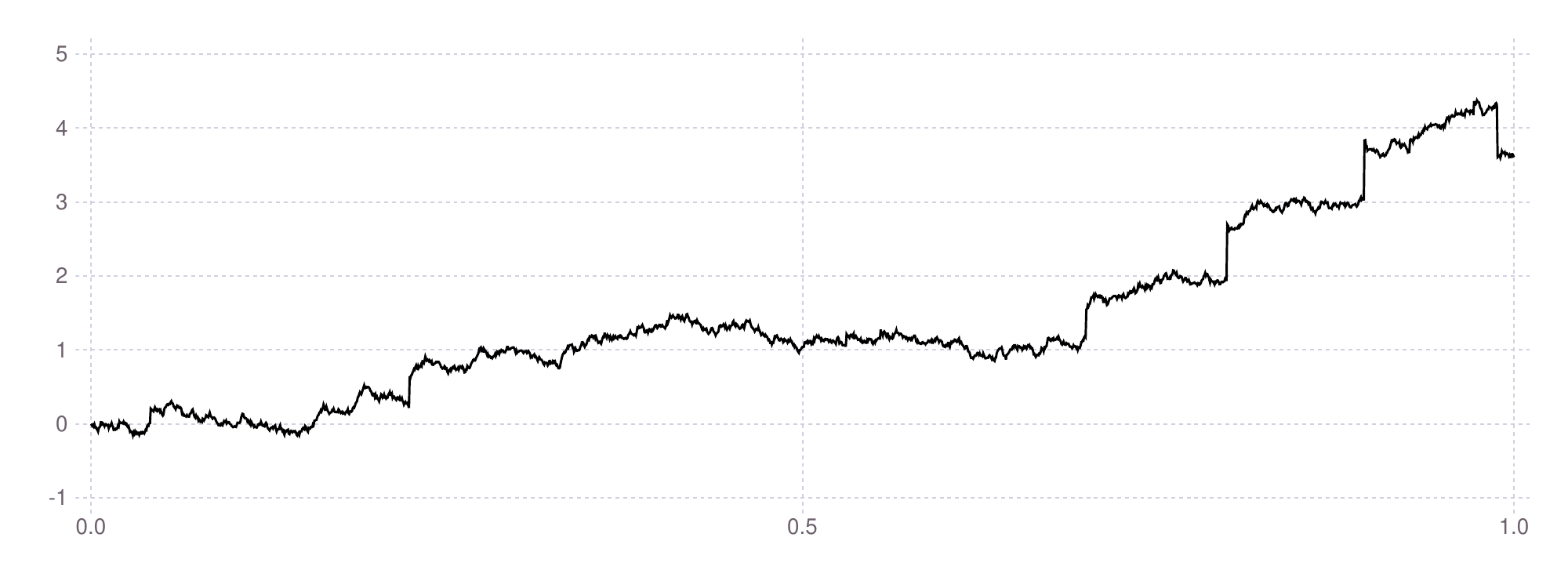}
\includegraphics[width=0.46\textwidth]{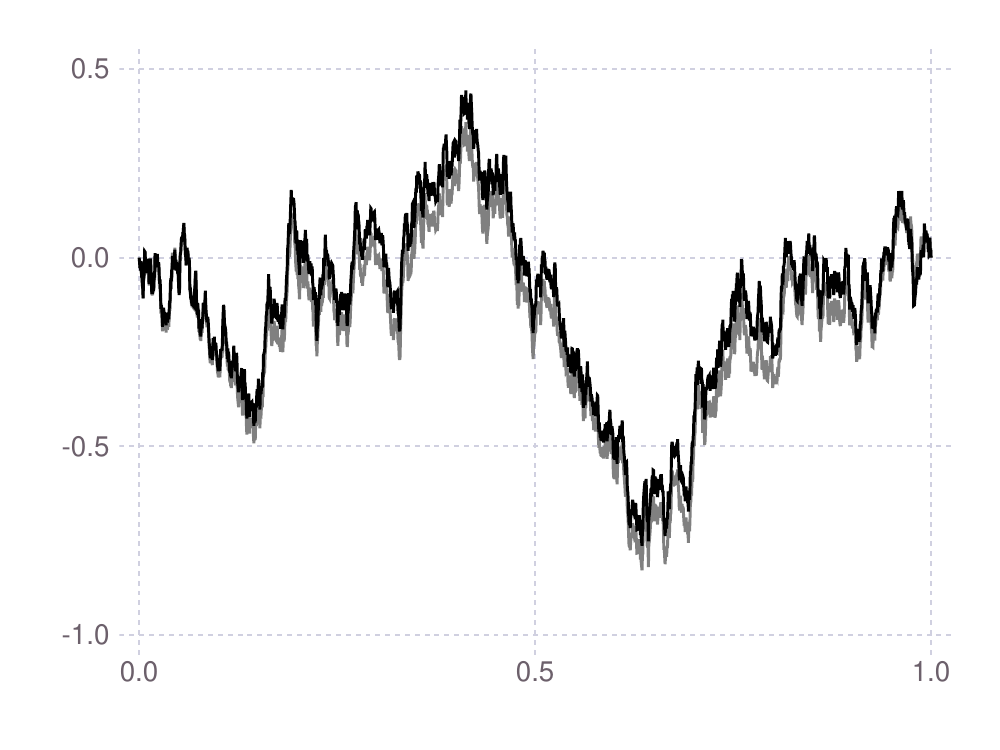}
\includegraphics[width=.52\textwidth]{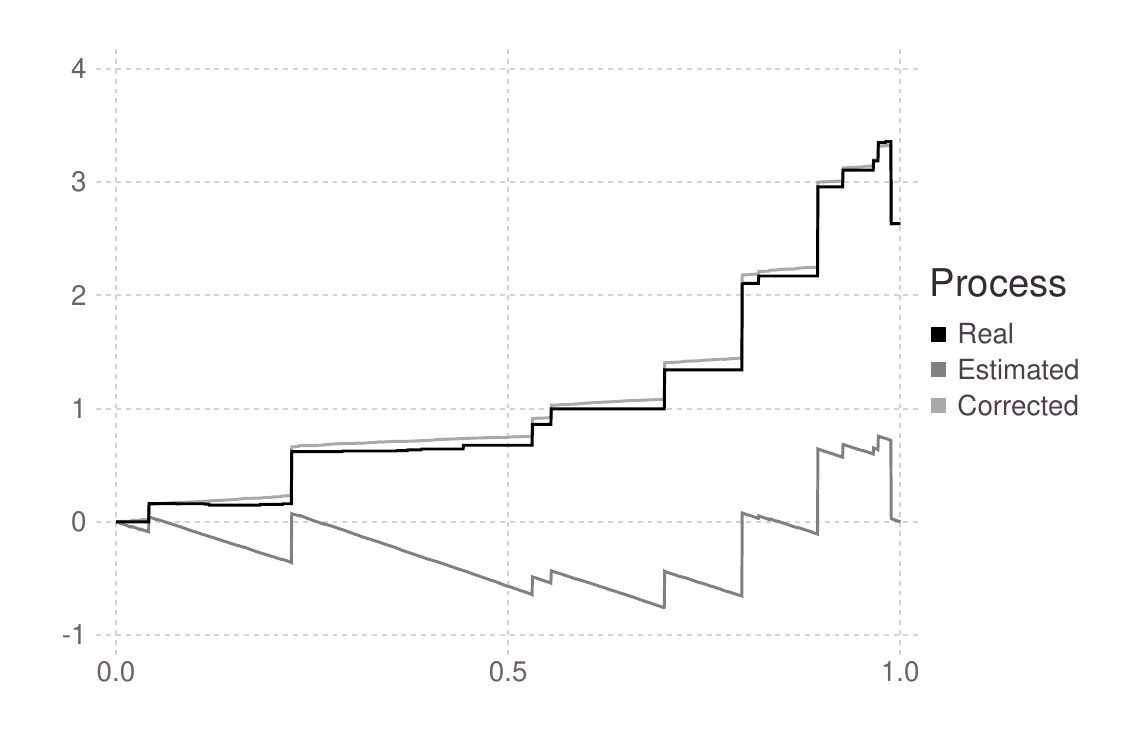}
\end{center}
\caption{\small The first picture depicts a path of $X$ for $\sigma=1$ and 
	$Y$ being variance gamma processes. The second picture 
	shows $(W_t-tW_1)_{t\in[0,1]}$ (black) and $(W^\n_t-tW^\n_1)_{t\in[0,1]}$ (grey) 
	using Method~I and  $n=10^4$. The third picture compares $Y$ and $X-W^\n$ 
	with corrected drift to match the endpoints.}
\label{fig:X}
\end{figure}

Let us provide some intuition. 
On the small time scale an overwhelming number of the increments of $X$ 
are close to those of $\sigma W$. By self-similarity the scaled increments 
of $W$ are i.i.d.\ standard normal random variables, whose empirical quantile 
function exhibits fast convergence due to light tails of the normal distribution. 
At an intuitive level this explains that 
ordering the increments of $W'$ according to the increments of $W$ or $X$ may 
produce a well-coupled process. Nonetheless, the result may still look surprising 
even for the purely Brownian case. In some sense, a path of the Brownian 
bridge is determined by the ordering of its infinitesimal increments.
Interestingly, in the case of no Brownian component ($\sigma=0$), despite the 
increments following the same order, the limiting result of this procedure is a 
standard Brownian motion independent of the original process~$X$, 
see Proposition~\ref{prop:nosigma} below. 


\subsection{Method II: threshold filter}\label{sec:meth2}
A much more
intuitive path decomposition method is based on a threshold filter. 
More precisely, we consider
\begin{equation}\label{eq:thresh_def}
\hat W^\n_t=\sigma^{-1}\sum_{i\leq tn, |\Delta^n_i X|\leq a_n}\Delta^n_i X,
\qquad t\in[0,1],
\end{equation}
which is the (scaled) skeleton that results from removing the increments of $X$ 
whose size is larger than $a_n>0$. Throughout this work 
we adopt the notation $\Delta^n_i X=X_{i/n}-X_{(i-1)/n}$ for $i=1,\ldots,n$,
which is standard in discretisation of processes~\cite{jacod}.

A wide range of work has been done on the subject of estimation via threshold 
filters, see~\cite[Ch.~9 \&~13]{jacod} 
and~\cite{MR2510946,MR2787615,MR2832818}. 
Most of these works, however, explore the recovery of a generalised 
path-variation or related quantities, and it seems that path decomposition of a 
L\'evy process has been overlooked in part. For instance, the conditions 
of~\cite[Thm~9.1.1]{jacod} are not satisfied by the identity function 
$f:x\mapsto x$ when $Y$ has infinite variation on compacts. Moreover, 
even in the covered cases, the speed of convergence seems not to be known, 
see~\cite[Thm~13.1.1]{jacod}. 

It is possible that this case has been 
overlooked since $\hat W^\n$ may explode when $Y$ has infinite variation on 
compacts. It is thus very important in that case to subtract the term 
$\hat W^\n_1 t$ as it plays a crucial compensating role. We will show that with 
the right choice of threshold $a_n$, we have
\[
\big(\hat W^\n_t-\hat W_1^\n t\big)_{t\in[0,1]}\,\cip\, \big(W_t-W_1 t\big)_{t\in[0,1]},
\]
see Theorem~\ref{thm:tunc_method} below, which also provides the rate of 
convergence. Once again, we avoid writing this convergence in the form 
of~\eqref{eq:meth1}, since then both sides may explode. However, if~$Y$ 
has bounded variation on compact intervals and possesses a linear drift 
$\gamma_0$ then 
\[
\big(\hat W^\n_t\big)_{t\in[0,1]}\,\cip\, \big(W_t+\gamma_0\sigma^{-1}t\big)_{t\in[0,1]},
\]
see Proposition~\ref{prop:drift}. 
For this method it may have been cleaner to avoid scaling by $\sigma^{-1}$, 
but we prefer to be consistent: $\hat W^\n$ and $W^\n$ both approximate 
the path of a standard Brownian motion. Note, however, that for any $n$, the 
approximation $W^\n$ is a discretely observed Brownian motion while 
$\hat W^\n$ need not be. We stress that for finite $n$ this method depends 
on the tuning parameter $a_n$, which may present a serious hurdle in applications. 
In this sense, our main ~Method~I is more robust.

\section{The main results: rates of convergence and limit laws}
Denote the L\'evy triplet (see~\cite[\S 2, Def.~8.2]{MR3185174}) 
of $X$ by $(\gamma,\sigma^2,\Pi(\D x))$ and write 
$\ov\Pi(x)=\Pi(\R\setminus(-x,x))$ for any $x>0$. 
Throughout this work we assume that $\sigma>0$ unless stated otherwise, 
and use the notation in~\eqref{eq:dec}.
The quality of decomposition of the path of $X$ crucially depends on the 
activity of small jumps. Therefore, we define two indices 
$0\leq \beta_*\leq\beta^*\leq 2$ capturing some main characteristics:
\begin{align*}
\beta^*&=\inf\Big\{p\geq 0:\int_{(-1,1)}|x|^p\Pi(\D x)<\infty\Big\},\\
\beta_*&=\inf\left\{p\geq 0:\liminf_{x\downarrow 0}x^p\ov\Pi(x)=0 \right\}.
\end{align*}
The index $\beta^*$ is known as the Blumenthal-Getoor index, 
whereas $\beta_*$ reminds Pruitt's index~\cite{pruitt}, which must lie 
between $\beta_*$ and $\beta^*$. 
Importantly, $\beta_*=\beta^*$ under some weak regularity assumptions, 
such as $\ov\Pi(x)$ being regularly varying at~$0$ with some index~$-\a$, 
in which case 
\begin{equation}\label{eq:RValpha}
\beta^*=\beta_*=\a,
\end{equation} 
a simple consequence of the standard theory of regular 
variation~\cite[\S 1]{BGT}.

\subsection{Method I}
As mentioned in \S\ref{sec:meth1} above, we consider a standard Brownian 
motion $W'$ independent of $X$ and~$W$. 
Let $\pi$ be the (random) permutation of the indices $1,\ldots,n$ such that the 
ordering of $\Delta^n_{\pi(\cdot)}W'$ coincides with that of $\Delta^n_{\cdot}X$. 
In other words, if $s$ is a permutation such that $\Delta^n_{s(\cdot)}X$ is an 
increasing sequence then $\Delta^n_{\pi(s(\cdot))}W'$ is also an increasing 
sequence. Such permutation $\pi$ is a.s.\ unique since there are a.s.\ no ties 
in either sequence. Finally, we take the corresponding partial sum process
\begin{equation}\label{eq:Wprime}
W^\n_t=\sum_{i\le nt}\Delta^n_{\pi(i)}W',\qquad t\in[0,1],
\end{equation}
which is a Brownian random walk. This can be seen by noting that the composition 
of a fixed permutation and a random uniform permutation is also uniformly 
distributed, and so the increments $\Delta^n_iW^{(n)}$ are i.i.d. indeed.
We may also keep the bridges in-between discretisation, which would then 
yield a standard Brownian motion. Note that the joint process 
$(W^\n_{i/n},W_{i/n})$, $i=1,\ldots,n$, has exchangeable increments 
but is not a random walk. 
Next we state the main result.


\begin{theorem}\label{thm:gral_rate}
For any $p\in(\beta^*,2]\cup\{2\}$ 
it holds that
\begin{equation}\label{eq:thm}
n^{(2-p)/4}\sup_{t\in [0,1]}\big|W_t-W_t^\n-(W_1-W^\n_1)t\big|\cip 0.
\end{equation}
Moreover,  this convergence fails for any $p\in[0,\beta_*)\cup\{0\}$.
\end{theorem}

It is noted that the final statement of Theorem~\ref{thm:gral_rate} implies that 
with some positive probability the quantity on the left hand side of~\eqref{eq:thm}
becomes arbitrarily large for some large~$n$. Thus we establish the exact 
convergence rate in the logarithmic sense when $\beta_*=\beta^*$ and, 
in particular, this rate is $n^{-(2-\a)/4}$ in the regularly varying 
case~\eqref{eq:RValpha}. Note as well that Theorem~\ref{thm:gral_rate} also 
implies the convergence of the bivariate approximation in~\eqref{eq:joint} 
with exactly the same rate. The convergence in~\eqref{eq:joint} requires 
access to the parameter~$\sigma$. However, this parameter may be 
estimated first without deteriorating the resulting convergence speed 
in~\eqref{eq:joint}. For instance, according to~\cite{MR2394762} 
(see also~\cite{jacod_reiss}), one can construct an estimator of~$\sigma$  
based on threshold filters (analogous to our Method II) with error decay 
$\Oh_\p(n^{-1/2})$ if $\beta^*<1$ and, otherwise, 
$\Oh_\p(n^{-(2-p)/2})$ for any $p\in(\beta^*,2]$. 


\subsection{Method I. Extensions}
\label{subsec:extensions}

\subsubsection{Dynamics under a dominating probability measure}
\label{subsec:ext-1}
It suffices to let $W$ be a Brownian motion independent of the pure-jump 
L\'evy process $Y$ under some probability measure $\Q$ dominating $\p$, 
that is, $\Q\gg\p$. Indeed, in that case the limit in Theorem~\ref{thm:gral_rate} 
holds under $\Q$ and thus, under $\p$. For example, by Girsanov's theorem, 
we may consider $W_t=B_t+\int_0^t f(s)\D s$, $t\in[0,1]$, where $f$ 
is an adapted process satisfying $\int_0^1 f(t)^2\D t<\infty$ 
and $B$ is a standard Brownian motion under $\p$. 
In fact, the process~$Y$ may also be a rather general pure-jump 
semimartingale under $\p$, possibly dependent on~$W$, 
see~\cite[Thm~2.3]{MR2083711} 
and~\cite[Thm~III.3.24, p.~172]{MR1943877}, but also Proposition~\ref{prop:CPP} below. 

\subsubsection{Extensions under exchangeability}
\label{subsec:ext-2}
Further generalisations are possible. Namely, the convergence 
in~\eqref{eq:thm} is guaranteed for any process $Y$ (possibly 
dependent on the Brownian motion $W$) such that the increments 
of the bivariate process $(W,Y)$ are exchangeable and $Y$ satisfies:
\[
n^{2-p/2}\e\left(\left(\Delta^n_1Y\right)^2
\wedge \frac{\log n}{n}\right)\to 0.
\]

It would be interesting to understand if exchangeability can be replaced 
by another structural assumption. One way is to ensure 
that~\eqref{eq:toprove} below is sufficient for the corresponding partial 
sums to vanish. In this regard we point out that a martingale 
assumption~\cite[Eq.~(2.2.35)]{jacod} seems to be of no immediate use 
because of inherent reorderings. 

\subsubsection{Multidimensional case}
Our method readily applies in a multivariate setting, where $Y$ is an 
$\R^d$-valued pure-jump \levy process, $W$ is an independent 
$d$-dimensional Brownian motion with standard but possibly correlated 
components and $\sigma$ is a diagonal scaling matrix. Indeed, 
by applying the decomposition procedure to every component of $X$ 
we may recover the entire path of the $d$-dimensional bridge 
$(W_t-t W_1)_{t\in[0,1]}$. One may use a single one-dimensional 
standard Brownian motion $W'$ for every component of~$X$, 
which may seem counter-intuitive as the dependence structure across 
coordinates is captured by the reorderings of increments.
In a degenerate case when the rank of the correlation matrix (suppose it is known)  is smaller 
than~$d$, we may use appropriate directions to reduce the number of 
required one-dimensional reconstructions to the given rank. 
Finally, we point out that the generalisations discussed above still 
apply in this context, see \S\ref{subsec:ext-1} in particular.

\subsection{Method I. Further results}
\label{subsec:Rd}
We have a more precise result in the case when $Y$ is a general piecewise 
constant process, including the compound Poisson process case. 
Note that one may always add a linear drift to $Y$ since this does not 
affect $W^\n$. The following result is stated for the discrete skeleton, 
since it may fail otherwise as the maximal deviation of $W$ from its 
discrete skeleton is of the same order:
\begin{equation}\label{eq:Wskeleton}
\sup_{t\in[0,1]}\big|W_t-W_{\lfloor tn\rfloor/n}\big|
=\Theta_\p\big(\sqrt{\log n/n}\big),
\end{equation}
meaning that the function on the right-hand side is, up to multiplicative constants, 
both an upper and lower asymptotic bound for the left-hand side. 
Throughout, the space $\mathcal D[0,1]$ of right-continuous functions 
with left-hand limits on $[0,1]$ is endowed with the standard Skorohod 
$J_1$-topology~\cite[Ch.\ 16 and A2]{MR1876169}. 
\begin{prop}\label{prop:CPP}
Let $(Y_t)_{t\in[0,1]}$ be a piecewise constant process 
(not necessarily L\'evy) independent of $W$ with jumps $J_1,\ldots,J_N$ 
at times $T_1,\ldots,T_N\in(0,1)$. Then the limit 
\[
\sqrt{\frac{n}{2\log n}}\left(W_{\lfloor tn\rfloor/n}-W_t^\n-(W_1-W^\n_1)t\right)
	\cip \sum_{i=1}^N \sign(J_i)(t-\1{T_i\leq t})
\]
holds in $\mathcal D[0,1]$.
\end{prop}

Interestingly, only the signs of the jump sizes $J_i$ appear in the limit.
In the purely Brownian case ($Y=0$) the effective convergence rate is 
$\sqrt{\log \log n/n}$, see Appendix~\ref{sec:normal}.

The final result covers the case $\sigma=0$. That is, we apply our 
procedure for a L\'evy process without Gaussian part: $X=Y$. 
Interestingly, this `coupling' yields, in the limit,  a Brownian motion 
independent of $X$. We believe that the following result is true even 
in the highest activity  case $\beta^*=2$, but its proof seems to 
require a more careful analysis. 
\begin{prop}\label{prop:nosigma}
If $\sigma=0$ and $\beta^*<2$, then the distributional convergence 
\[
(X_t,W^\n_t)\convd (X_t,B_t)
\]
holds in $\mathcal D[0,1]$, where $B$ is a standard Brownian motion 
independent of~$X$. 
\end{prop}

\subsection{Method I. Numerical illustration}
\label{subsec:num1}
We conclude the discussion of Method I with numerical illustrations of 
Theorem~\ref{thm:gral_rate} and Proposition~\ref{prop:CPP}. 
For this example we assume $\sigma=1$, $W$ is a 3-dimensional Bessel 
process (a Brownian motion under an equivalent probability measure) 
and $Y$ is an independent strictly $\a$-stable process. We consider 
various values of $\a$ and, as the other parameters are less relevant, 
we fix the skewness parameter at $\beta=0.5$ and take unit scale. 
We work with 5 approximation levels $n=10^3,10^4,\ldots,10^7$ and 
for each scenario we compute the maximal difference between the 
discretised bridge on a uniform grid of $N=10^7$ points and its 
approximation at level $n$: 
\begin{equation}\label{eq:error}
\sup_{i\leq N}\Big|W_{i/N}-\tfrac{i}{N}W_1
	-\big(W_{i/N}^\n-\tfrac{i}{N}W^\n_1\big)\Big|.
\end{equation}

We point out that we do not resample $W'$, i.e. we use the same path 
for each of the resolution levels $n$. We replicate the procedure $100$ 
times to estimate the expected value of the quantity in~\eqref{eq:error} 
and its standard deviation, see Table~\ref{table:results}. For the sake 
of comparison, we also take $n=1$, where $W^\n=W'$ 
and~\eqref{eq:error} quantifies the discrepancy between two 
independent bridges, resulting in $1.202(.3333)$.

\begin{table}[h]
	\begin{tabular}{c|c|c|c|c|c|c}
		& $\a=0.2$ & $\a=0.6$ & $\a=1$ & $\a=1.4$ & $\a=1.8$ & $\a=1.99$\\
		\hline
		$n=10^3$& .1590(.0430) & .1704(.0438) & .1912(.0598) & 
								.2330(.0533) & .3115(.0825) & .3399(.0786)\\
		$n=10^4$& .0626(.0175) & .0741(.0199) & .0980(.0241) & 
								.1468(.0415) & .2623(.0715) & .3061(.0793)\\
		$n=10^5$& .0243(.0086) & .0326(.0083) & .0536(.0137) & 
								.1002(.0306) & .2275(.0649) & .2958(.0793)\\
		$n=10^6$& .0090(.0033) & .0145(.0036) & .0290(.0079) & 
								.0704(.0221) & .2007(.0551) & .2905(.0811)\\
		$n=10^7$& .0031(.0017) & .0056(.0018) & .0153(.0043) & 
								.0501(.0160) & .1773(.0511) & .2886(.0821)
	\end{tabular} 
\vspace{0.2cm}
	\caption{Means (and standard deviations) of the errors given 
		by~\eqref{eq:error} when $Y$ is an $\a$-stable process.}
	\label{table:results}
\end{table}

\begin{figure}[h!]
	\begin{tikzpicture} 
	\begin{loglogaxis} 
	[
	ymin=.001,
	ymax=.5,
	xmin=1000,
	xmax=10000000,
	xlabel={\small $n$},
	ylabel={\small error},
	width=11cm,
	height=5.5cm,
	axis on top=true,
	axis x line=bottom, 
	axis y line=left,
	axis line style={->},
	x label style={at={(axis description cs:0.8,0.2)},anchor=north},
	legend style={at={(.2,.015)},anchor=south east}
	]
	
	\addplot[only marks, solid, mark=o, mark options={scale=.85, solid}, color=black,
	]
	coordinates {
		(1000.0,0.158987)(10000.0,0.0626097)(100000.0,0.0243297)
		(1000000,0.0090032)(10000000,0.003064)
	};
	\addplot [solid, samples = 1000, color=black, line width = .3, 
		domain=1000:10000000]
	{(x/100000)^(-(2-.2)/4)*0.0243297};
	
	\addplot[only marks, solid, mark=o, mark options={scale=.85, solid}, color=gray!140,
	]
	coordinates {
		(1000.0,0.170403)(10000.0,0.0740925)(100000.0,0.0326138)
		(1000000,0.0144522)(10000000,0.00561922)
	};
	\addplot [solid, samples = 1000, color=gray!140, line width = .3, 
		domain=1000:10000000]
	{(x/100000)^(-(2-.6)/4)*0.0326138};
	
	\addplot[only marks, solid, mark=o, mark options={scale=.85, solid}, color=gray!125,
	]
	coordinates {
		(1000.0,0.191159)(10000.0,0.0980456)(100000.0,0.0536214)
		(1000000,0.029043)(10000000,0.0153297)
	};
	\addplot [solid, samples = 1000, color=gray!125, line width = .3, 
		domain=1000:10000000]
	{(x/100000)^(-(2-1)/4)*0.0536214};
	
	\addplot[only marks, solid, mark=o, mark options={scale=.85, solid}, color=gray!110,
	]
	coordinates {
		(1000.0,0.232989)(10000.0,0.14678)(100000.0,0.100212)
		(1000000,0.0703617)(10000000,0.0500992)
	};
	\addplot [solid, samples = 1000, color=gray!110, line width = .3, 
		domain=1000:10000000]
	{(x/100000)^(-(2-1.4)/4)*0.100212};
	
	\addplot[only marks, solid, mark=o, mark options={scale=.85, solid}, color=gray!95,
	]
	coordinates {
		(1000.0,0.311508)(10000.0,0.262276)(100000.0,0.227457)
		(1000000,0.200695)(10000000,0.17725)
	};
	\addplot [solid, samples = 1000, color=gray!95, line width = .3, 
		domain=1000:10000000]
	{(x/100000)^(-(2-1.8)/4)*0.227457};
	
	\addplot[only marks, solid, mark=o, mark options={scale=.85, solid}, color=gray!80,
	]
	coordinates {
		(1000.0,0.33991)(10000.0,0.306071)(100000.0,0.295824)
		(1000000,0.290505)(10000000,0.288612)
	};
	\addplot [solid, samples = 1000, color=gray!80, line width = .5, domain=1000:10000000]
	{(x/100000)^(-(2-1.99)/4)*0.295824};
	
	\legend {\small $\alpha=0.2$,,
		\small $\alpha=0.6$,,
		\small $\alpha=1.0$,,
		\small $\alpha=1.4$,,
		\small $\alpha=1.8$,,
		\small $\alpha=1.99$};
	\end{loglogaxis}
	\end{tikzpicture}
	\caption{\small Log-log plot of data in Table~\ref{table:results} with lines of slope $-(2-\a)/4$ depicting the theoretical rate in Theorem~\ref{thm:gral_rate}}
	\label{fig:log_table}
\end{figure}
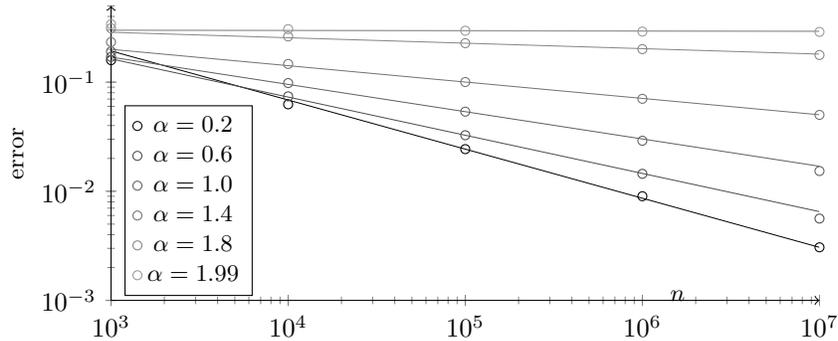
Figure~\ref{fig:log_table} provides the log-log plot together with 
lines corresponding to the theoretical rates given by 
Theorem~\ref{thm:gral_rate}. That is, the lines pass through the 
given value at $n=10^5$ and their slopes are given by $-(2-\a)/4$.

Next assume $\sigma$ and $W$ are as in the first paragraph of 
\S\ref{subsec:num1} and $Y$ is a Poisson process with intensity $3$. 
The Figure~\ref{fig:P} below exemplifies the limit established in 
Proposition~\ref{prop:CPP} above. Note how the signs of the jumps 
in the limit are opposite to those of $Y$.

\begin{figure}[h]
	\begin{center}
		\includegraphics[width=0.99\textwidth]{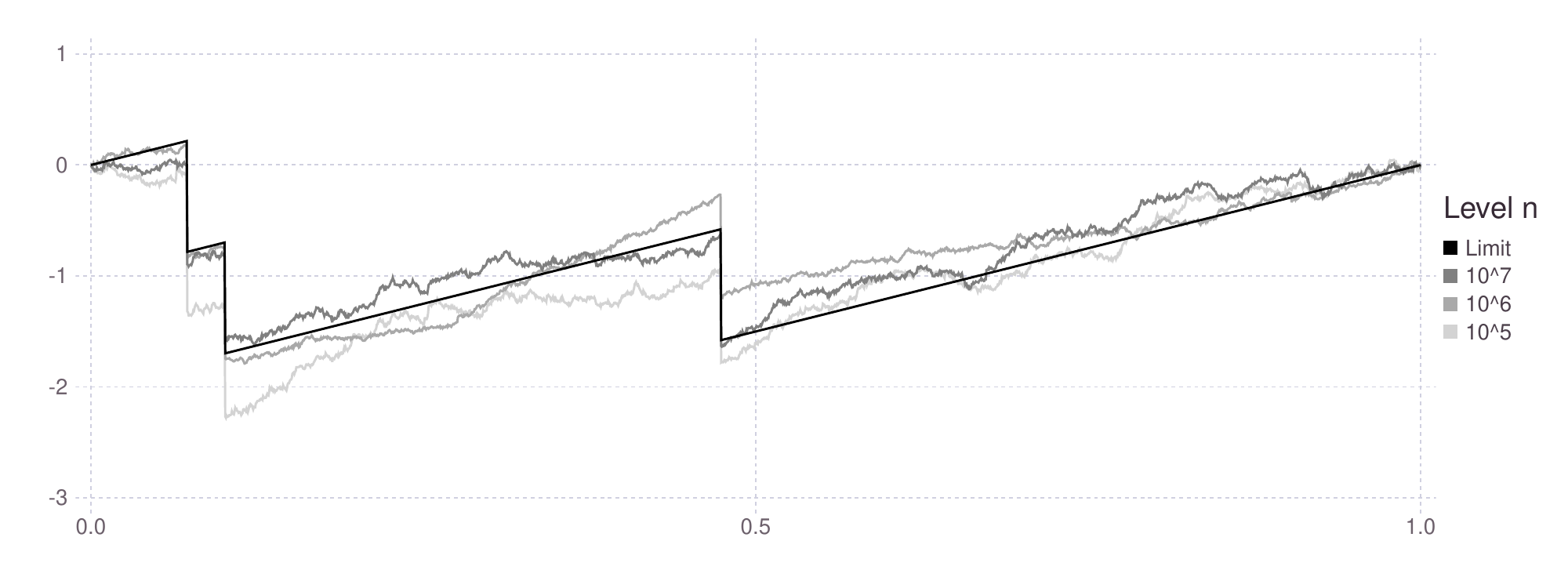}
	\end{center}
	\caption{\small Convergence in Proposition~\ref{prop:CPP} 
		for the approximation levels $n\in\{10^5,10^6,10^7\}$.}
	\label{fig:P}
\end{figure}

\subsection{Method II}
Finally, we turn our attention to the filtering method described 
in \S\ref{sec:meth2}. Here, the main approximant is the process 
$\hat W_t^\n$ defined in~\eqref{eq:thresh_def}.

\begin{theorem}
\label{thm:tunc_method}
For any $p\in(\beta^*,2]$ it holds that
\begin{equation}
\label{eq:trunc}
n^{(2-p)/4}\sup_{t\in [0,1]}
	\big|W_t-\hat W_t^\n-(W_1-\hat W_1^\n)t\big|\cip 0,
\end{equation}
assuming that the sequence $a_n$ satisfies 
\begin{equation}\label{eq:a_n}
\liminf_{n\to\infty}\frac{na^2_n}{\sigma^2\log n}\geq 2-\beta^*
\qquad\text{and}\qquad
n^{1/2-\epsilon}a_n\to 0\quad\text{ for all } \epsilon>0.
\end{equation}
The limit~\eqref{eq:trunc} also holds for $p=2$ and any 
$\beta^*$ if $a_n\to 0$ and 
$\liminf_{n\to\infty}na^2_n/\log n>0$. 
\end{theorem}

In words, the threshold $a_n$ should (roughly) be of order 
$n^{-1/2}$, but not smaller than $c\sqrt{\log n/n}$ for a certain 
constant $c>0$. The choice $a_n=n^{-1/2}\log n$, for example, 
satisfies~\eqref{eq:a_n} in all cases. Both upper and lower bounds 
on $a_n$ can be relaxed, but then the rate would deteriorate. 
The allowed relaxations of the bounds and the corresponding 
rates are stated in Lemma~\ref{lem:trunc1} of~\S\ref{sec:proofsII}. 
In some sense, the rate is much less sensitive to increasing $a_n$ 
than to decreasing it. 

Importantly, the same (logarithmic) rate can be obtained using 
filtering method with an appropriate~$a_n$. We expect, however, 
that the choice of such threshold for a finite~$n$ may have a serious 
impact on the quality of decomposition, see also~\cite{MR2510946}.
Our main Method I has no such issues.
Let us also supplement Theorem~\ref{thm:tunc_method} with the 
negative results showing that the stated threshold bounds and the 
convergence rates are optimal.

\begin{prop}\label{prop:neg}
The following statements are true
\begin{itemize}
\item[\normalfont(a)] If $p\in[0,\beta_*)$ then the limit~\eqref{eq:trunc} fails.
\item[\normalfont(b)] If 
$\liminf_{n\to\infty} na^2_n/(\sigma^2\log n)<2-\beta^*$ then 
the limit~\eqref{eq:trunc} fails for some $p\in(\beta^*,2]$. 
\item[\normalfont(c)] If $\Pi\ne0$, $a_n\to 0$ and 
$n^{1/2-\epsilon}a_n\not\to 0$ for some $\epsilon>0$ then 
the limit~\eqref{eq:trunc} fails for some $p\in(\beta_*,2]$.
\end{itemize}
\end{prop}

As mentioned above, compensation by $\hat W^\n_1 t$ is not 
required in the case when $Y$ has bounded variation on compact 
intervals (implying $\beta^*\in[0,1]$). We have the following 
additional result: 
\begin{prop}
\label{prop:drift}
If $p\in(\beta^*,1]$ and $a_n$ satisfies~\eqref{eq:a_n}, 
it holds that
\begin{equation}
\label{eq:trunc'}
n^{(1-p)/2}\sup_{t\in [0,1]}\Big|W_t+\gamma_0\sigma^{-1}t-\hat W_t^\n\Big|\cip 0,
\end{equation}
where $\gamma_0$ is the linear drift of~$X$. 
The limit~\eqref{eq:trunc'} is also true for $p=1$ if $Y$ has 
bounded variation on compact intervals, $a_n\to 0$ and 
$\liminf_{n\to\infty}na^2_n/\log n>\sigma^2$.
\end{prop}
Observe that in the case considered by Proposition~\ref{prop:drift} we may 
recover $\sigma W_t+\gamma_0t$ and the pure jump part $Y_t-\gamma_0 t$, 
whereas separating the linear drift from the Brownian motion is impossible. 
The convergence rate, however, is worse than in Theorem~\ref{thm:tunc_method}.

\subsection{Methods I and II. Numerical comparison}
\label{num_comparison}

To compare both methods, we test them on the process $X=\sigma W+Y$ 
where $\sigma=1$, $W$ is a $3$-dimensional Bessel process and $Y$ is a 
strictly $\a$-stable process with $\a=1.2$ and skewness parameter $-0.5$. 
In the application of Method II we used the threshold $a_n=n^{-1/2}\log n$, 
which satisfies~\eqref{eq:a_n} independently of $\sigma$ and $\alpha$. 

Figure~\ref{fig:comparison} displays the errors 
$W_t-W^\n_t -(W_1-W^\n_1)t$ and $W_t-\hat W^\n_t -(W_1-\hat W^\n_1)t$ 
under both methods. Table~\ref{table:comparison} reports the mean and 
standard deviation of the maximal absolute value (on the skeleton) of the 
two error processes after repeating the procedure described in the previous 
paragraph for 1000 independent paths of $X$. The comparison of the test 
statistics in Table~\ref{table:comparison} suggests that Method I outperforms 
Method II by a constant factor in this example. 

\begin{figure}[h]
	\begin{center}
		\includegraphics[width=0.49\textwidth]{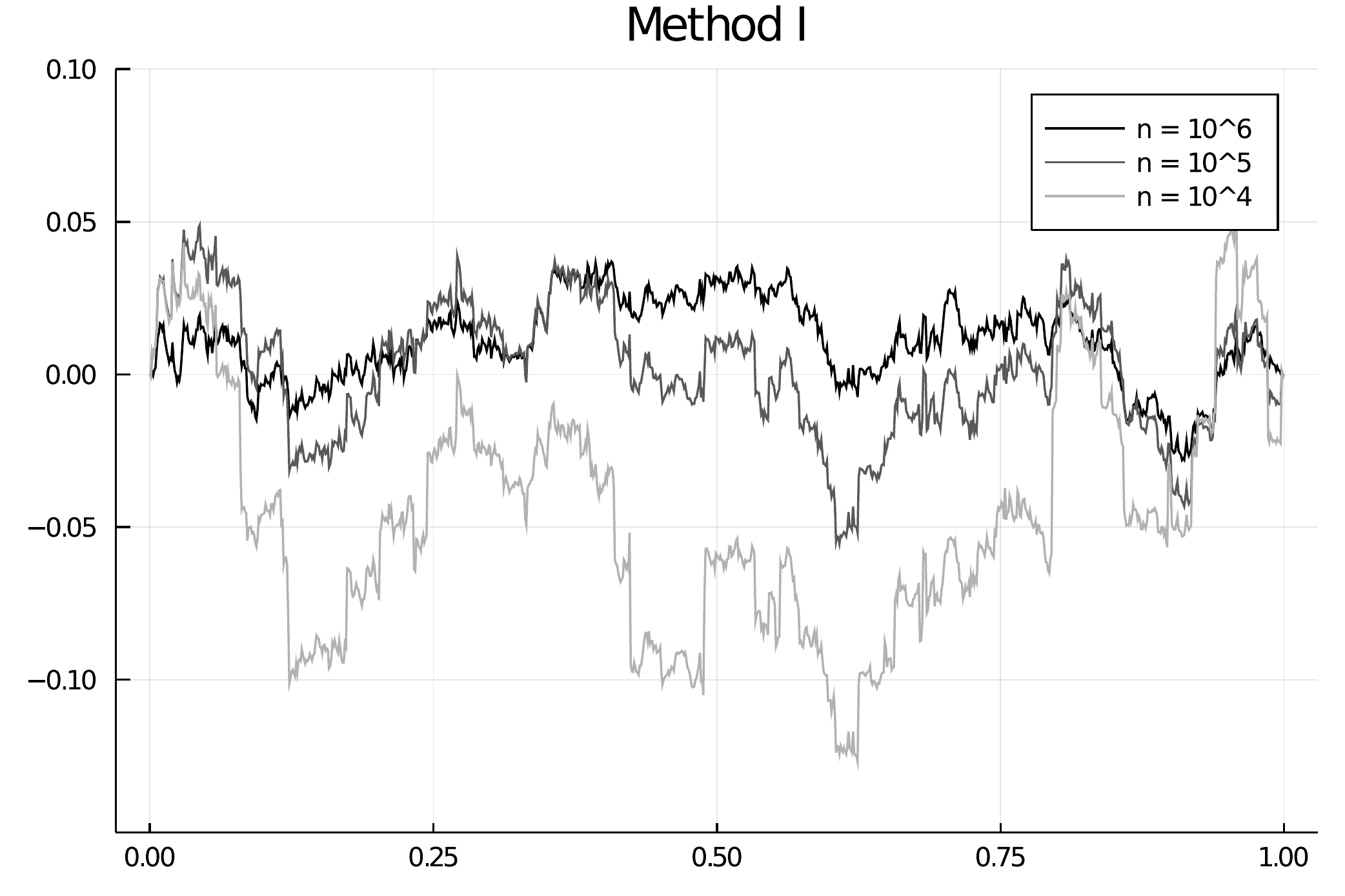}
		\includegraphics[width=0.49\textwidth]{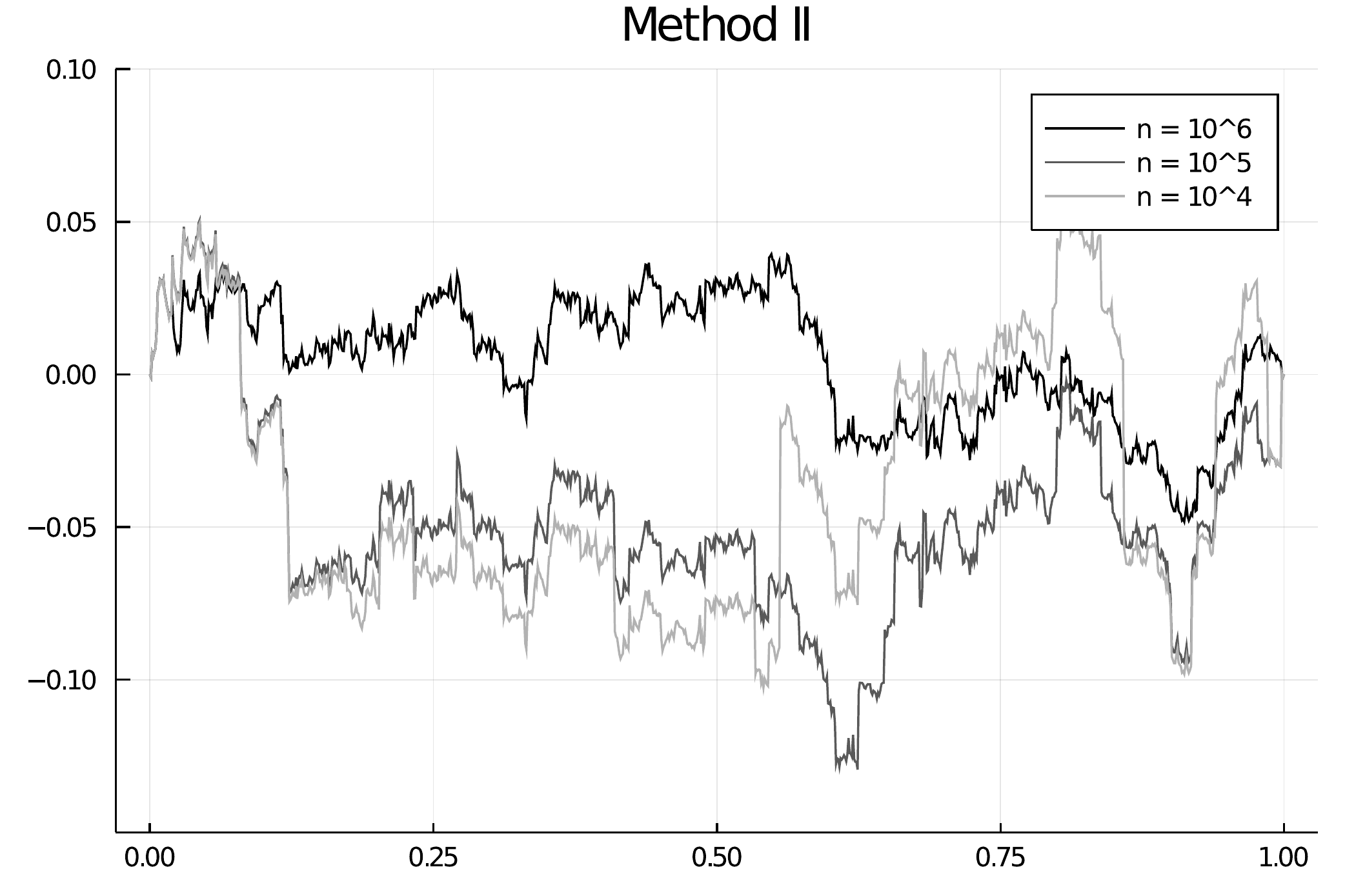}
	\end{center}
	\caption{\small Errors $W_t-W^\n_t -(W_1-W^\n_1)t$ and 
		$W_t-\hat W^\n_t -(W_1-\hat W^\n_1)t$ for the approximation 
		levels $n\in\{10^4, 10^5, 10^6\}$.}
	\label{fig:comparison}
\end{figure}

\begin{table}[h]
	\begin{tabular}{c|c|c|c|c}
		Level $n$ & $n=10^3$ & $n=10^4$ & $n=10^5$ & $n=10^6$ \\
		\hline
		Method I & .2299(.0789) & .1542(.0488) & .1013(.0317) & .0660(.0203) \\
		Method II & .2363(.0837) & .1746(.0570) & .1214(.0371) & .0842(.0257) \\
	\end{tabular} 
	\vspace{0.2cm}
	\caption{Means (and standard deviations) of the maximal absolute value of the 
		errors.}
	\label{table:comparison}
\end{table}

\section{Proofs for Method I}\label{sec:proofs}
Consider the process of interest 
\[
\Big(n^{(2-p)/4}\big(W_t-W_t^\n-(W_1-W^\n_1)t\big)\Big)_{t\in[0,1]},
\] 
with $p\in(0,2]$ and let $\xi_{ni}$ be its $i$th increment, that is, we 
apply~$\Delta_i^n$. According to~\eqref{eq:Wskeleton} we may restrict our 
attention to the partial sum process $(\sum_{i\leq tn}\xi_{ni})_{t\in[0,1]}$. 
Observe that 
\begin{align*}
\xi_{ni}
&=n^{(2-p)/4}\bigg(\Delta_i^nW-\Delta_{\pi(i)}^nW'
-\frac{1}{n}\sum_{j\le n}(\Delta_j^nW-\Delta_j^nW')\bigg)\\
&=n^{-p/4}\bigg(Z_i-Z'_{\pi(i)}-\frac{1}{n}\sum_{j\le n}(Z_j-Z'_j)\bigg),
\end{align*}
where $Z_1,\ldots,Z_n$ and $Z'_1,\ldots,Z'_n$ are i.i.d.\ standard normal variables. 
Recall that~$\pi$ is the permutation such that $Z'_{\pi(i)}$ has the same ordering as 
\begin{equation}\label{eq:order}
Z_i+\frac{1}{\sigma}\sqrt n \Delta_i^nY.
\end{equation}
Additionally, we define the permutation $\nu$ so that $Z_{\nu(i)}$ is ordered 
according to~\eqref{eq:order} and thus the orderings of $Z_{\nu(i)}$ and 
$Z'_{\pi(i)}$ coincide. In the decomposition 
 \[
 \xi_{ni}=\tilde\xi_{ni}+\hat\xi_{ni}
 	=n^{-p/4}\bigg(Z_{\nu(i)}-Z'_{\pi(i)}-\frac{1}{n}\sum_{j\le n}(Z_j-Z'_j)\bigg)
 		+n^{-p/4}(Z_i-Z_{\nu(i)}),
 \]
the second term does not depend on $W'$, whereas the first term essentially 
corresponds to comparing certain order statistics (the order is random and 
dependent on~$X$). 

The strategy is to split the analysis of the partial sum process of $\xi_{ni}$ into 
that of the partial sum processes of $\tilde\xi_{ni}$ and $\hat\xi_{ni}$. 
Importantly, $(\hat\xi_{ni})_{i=1,\ldots,n}$ and $(\tilde\xi_{ni})_{i=1,\ldots,n}$ 
are both exchangeable. In fact, this is true for any process $Y$ as long as 
$(\Delta_i^n W,\Delta_i^n X)$ is exchangeable. Thus the general theory 
in~\cite[Thm~3.13]{MR2161313} for exchangeable increment processes is 
applicable. In this respect, since 
$\sum_{i\le n}\hat \xi_{ni}=\sum_{i\le n}\tilde \xi_{ni}=0$, we note that the 
convergence in probability of the partial sum processes of $\tilde\xi_{ni}$ and 
$\hat\xi_{ni}$ to 0 is equivalent to, respectively, the limits 
\begin{equation}\label{eq:toprove}
\sum_{i\le n}\tilde \xi_{ni}^2\cip 0 
\qquad\text{and}\qquad
\sum_{i\le n}\hat \xi_{ni}^2\cip 0.
\end{equation}
Lemma~\ref{lem:normal} below establishes the first limit for any $p>0$.
The second convergence depends on the choice of $p$: 
it holds for large enough $p$ and fails for sufficiently small $p$. 

\subsection{Preparatory results}
\begin{lemma}\label{lem:normal}
Let $Z_{(1)}<\cdots<Z_{(n)}$ and $Z'_{(1)}<\cdots<Z'_{(n)}$ be two independent 
ordered sequences of $n$ standard normal random variables. Then the limit 
\[
a_n\sum_{i\le n} \Big(Z_{(i)}-Z'_{(i)}-\frac{1}{n}\sum_{j\le n}(Z_{(j)}-Z'_{(j)})\Big)^2 \cip0
\]
holds whenever $a_n\log\log n\to 0$.
\end{lemma}

\begin{proof}
Letting $\mu_n$ be the inner sum of the statement we note that
\[
\sum_{i\le n} \Big(Z_{(i)}-Z'_{(i)}-\frac{1}{n}\mu_n\Big)^2
=\sum_{i\le n} \Big(Z_{(i)}-Z'_{(i)}\Big)^2-\frac{1}{n}\mu_n^2.
\]
Since $\mu_n/\sqrt{n}\sim N(0,2)$ has constant distribution we have 
$\mu_n^2/(n\log\log n)\cip 0$. Thus it is left to prove that 
\[
a_n\sum_{i\le n} \Big(Z_{(i)}-Z'_{(i)}\Big)^2
=a_nn\|F_n^{-1}-G_n^{-1}\|_2^2 \cip 0,
\]
where $\|f\|_2^2=\int_0^1 f^2(x)dx$, and  $F_n^{-1}$ and $G_n^{-1}$ are the right-inverses of the empirical distributions of 
$Z_\cdot$ and $Z'_\cdot$, respectively.

Let $\Phi$ denote the standard normal distribution and let 
$\tilde{F}_n=\sqrt n(F_n^{-1}-\Phi^{-1})$ and 
$\tilde{G}_n=\sqrt n(G_n^{-1}-\Phi^{-1})$
be the respective normalised empirical quantile processes. 
By Minkowski inequality we have
\begin{align*}
n\|F_n^{-1}-G_n^{-1}\|_2^2=\|\tilde{F}_n-\tilde{G}_n\|_2^2\leq 2(\|\tilde{F}_n\|_2^2+\|\tilde{G}_n\|_2^2).
\end{align*}
From~\cite[Thm~4.6(ii)]{MR2121458} it can be deduced that 
$\|\tilde{F}_n\|_2^2/\log\log n\cip 1$. The same is true of $\tilde{G}_n$, 
completing the proof.
\end{proof}
It can be shown that $\log\log n$ is the ``right'' scale, see Appendix~\ref{sec:normal}. 
The following permutation Lemma (with $q=2$) is crucial to upper bound 
$\sum_{i\le n}(Z_i-Z_{\nu(i)})^2$. 
In this lemma it is more convenient to swap $\nu$ for $\nu^{-1}$.
\begin{lemma}\label{lem:permutation}
Let $z_1\leq \cdots\leq z_n$ be $n\geq 1$ ordered real numbers. For arbitrary 
$y_i\in\R$ consider a permutation $\nu$ such that $z_{\nu^{-1}(i)}+y_{\nu^{-1}(i)}$ 
is ordered. Then we have 
\[
\sum_{i\le n}|z_i-z_{\nu(i)}|^q\leq 2^q\sum_{i\le n} (|y_i|^q\wedge m^q),
\qquad q\ge 1,
\]
with $m=z_n-z_1$. 
\end{lemma}	

\begin{proof}
Without loss of generality we assume that $\nu$ has exactly one cycle, since 
otherwise we just sum over the cycles and increase the respective $m$ if needed. 
Moreover, the result is trivial for a cycle of length 1. 

It is a basic fact that 
\begin{equation}\label{eq:basic}
|z_i-z_j|^q\leq (|y_i|+|y_j|)^q\leq 2^{q-1}(|y_i|^q+|y_j|^q)
\end{equation}
whenever $i<j$ and $\nu(i)>\nu(j)$ or $i>j$ and $\nu(i)<\nu(j)$, i.e., if the order 
is flipped. Furthermore, this bound is still true when $|y_\cdot|$ is replaced by 
$|y_\cdot|\wedge m$.  

We call the ordered sequence $\nu(i),\nu^2(i),\ldots$ the successors of $i$. 
For each $i$ satisfying $i<\nu(i)$, we define:
\[
b(i)\text{ is the first successor of }i\text{ such that }\nu(b(i))<\nu(i)\leq b(i),
\]
and note that $b(i)$ is well defined, see Example~\ref{ex:permutation}.
From~\eqref{eq:basic} we have the bound:
\begin{align*}
|z_{i}-z_{\nu(i)}|^q\leq |z_{i}-z_{b(i)}|^q
\leq 2^{q-1}(|y_{i}|^q\wedge m^q+|y_{b(i)}|^q\wedge m^q).
\end{align*}
The case $i>\nu(i)$ is analogous but with inequalities reversed in the 
definition of $b(i)$. By summing up over all $i$ we get the upper bound for 
$\sum_{i\le n}|z_i-z_{\nu(i)}|^q$. This bound needs to be reduced since the 
same $b$ may appear multiple times. 

Suppose $i_1,\ldots,i_k$ with $k>1$ are all the indices with
\[
b^*=b(i_1)=\cdots=b(i_k).
\]
Without loss of generality we assume that $b^*>\nu(b^*)$ and so $i_j<\nu(i_j)$ 
for all $j=1,\ldots,k$. Moreover, let the numbering be such that the path from 
$i_1$ to $b^*$ passes through $i_2,\ldots,i_k$ in this order. Note that 
$i_2<\nu(i_1)$ implies that $b(i_1)$ occurs before $i_2$, a contradiction. 
Thus we have
\[
i_1<\nu(i_1)\leq i_2<\nu(i_2)\leq \cdots<i_k<\nu(i_k)\leq b^*.
\]
Hence, it holds that 
\[
|z_{i_1}-z_{\nu(i_1)}|^q+\cdots+|z_{i_k}-z_{\nu(i_k)}|^q\leq |z_{i_1}-b^*|^q,
\]
implying that only one term 
$2^{q-1}(|y_{i_1}|^q\wedge m^q+|y_{b^*}|^q\wedge m^q)$ out of $k$ is 
necessary. The proof is now complete. 
\end{proof}

Note that the constant $2^q$ in front of the upper bound can not be 
reduced in general. For example, let $q=n=2$ and 
$z_1=0,z_2=1,y_1=1/2+\epsilon,y_2=-y_1$ with some $\epsilon>0$. 
Then $z_1+y_1>z_2+y_2$ and the bound reads $2\leq 2(1+2\epsilon)^2$.

\begin{example}\label{ex:permutation}\rm
	Consider the permutation: $1\to 2\to 4\to 3\to 5\to 1$. 
	The summary of indices is given below:	
	\begin{center}
		\begin{tabular}{c|c|c|c}
			$i$ & direction &$b(i)$&$\# y_i$ in the bound\\
			\hline
			$1$ & $\rightarrow$ &5&2\\
			$2$ & $\rightarrow$ &4&1\\
			$3$ & $\rightarrow$ &5&1\\
			$4$ & $\leftarrow$ &3&2\\
			$5$ & $\leftarrow$ &1&2\\		
		\end{tabular}  
	\end{center}
	Note that the pair $(3,5)$ was not used in the construction of our bound.
\end{example}

Finally, we need some estimates for the L\'evy processes $Y$.
\begin{lemma}\label{lem:levy}
The following statements hold for any L\'evy process $Y$ without 
Brownian component:
\begin{itemize}
\item[\normalfont(a)] For any positive decreasing sequence $a_n\downarrow 0$ 
satisfying $a_n\sqrt n\to\infty$, we have the limit $\p(|Y_{1/n}|>a_n)\to 0$ 
and, for sufficiently large~$n$, the following bound holds: 
\[
n\p(|Y_{1/n}|>a_n)\ge\frac{1}{2}\ov\Pi(2a_n).
\]
\item[\normalfont(b)] For any $p\in(\beta^*,2]\cup\{2\}$, we have 
\[
n^{2-p/2}\e\left(Y_{1/n}^2\wedge \frac{\log n}{n}\right)\to 0.
\]
\item[\normalfont(c)] If $\beta^*<2$ then 
\[
\sqrt {n\log n}\e(|Y_{1/n}|\wedge 1)\to 0.
\]
\end{itemize}
\end{lemma}

The proof is based on some standard techniques and is deferred to 
\S\ref{sec:levy} in order to keep the presentation focused.

\subsection{Proofs of the main results}
In the following we say that events $(A_n)_{n\in\N}$ have high probability 
(for all large $n$) if $\p(A_n)\to 1$. Clearly, any finite collection of events 
with high probability 
jointly have 
high probability.

\begin{proof}[Proof of Theorem~\ref{thm:gral_rate}]
Recall that it is left to consider the quantities in~\eqref{eq:toprove}.
Lemma~\ref{lem:normal} implies that $\sum_{i\le n}\tilde\xi^2_{ni}\cip 0$ 
for any $p>0$ because the sum can be reordered so that both $Z$ and $Z'$ 
appear in increasing order. 
Thus is is left to (i) show $\sum_{i\le n}\hat\xi^2_{ni}\cip 0$ for $p>\beta_*$ 
(and $p=2$ when $\beta^*=2$), and (ii) to disprove this for $p\in(0,\beta_*)$.
For $p=0$ the convergence in~\eqref{eq:thm} always fails as a consequence of presence of $n$ independent scaled Brownian bridges between the grid points, see also~\eqref{eq:Wskeleton}.
\bigskip

Part (i).
By standard extreme value theory~\cite[(3.65)]{EMK} we have
\[
M_n-2\sqrt{2\log n}\cip 0, 
\qquad \text{where}\quad
M_n=\max_{i\le n} Z_i-\min_{i\le n} Z_i.
\]
According to~\eqref{eq:order} and Lemma~\ref{lem:permutation} there is the bound
\begin{equation*}
\sum_{i\le n}\hat\xi^2_{ni}
=n^{-p/2}\sum_{i\le n} \big(Z_i-Z_{\nu(i)}\big)^2
\leq 4n^{-p/2}\sum_{i\le n}\left(\frac{n}{\sigma^2}\left(\Delta^n_iY\right)^2
\wedge M^2_n\right).
\end{equation*}
With high probability $M^2_n<9\log n$ for all large~$n$. 
Moreover, by Lemma~\ref{lem:levy}(b),
\[
\e\bigg[ n^{-p/2}\sum_{i\le n}\Big(n\big(\Delta^n_iY\big)^2
	\wedge \log n\Big)\bigg]
=n^{1-p/2}\e\Big(nY_{1/n}^2
	\wedge \log n\Big)\to 0
\]
whenever $p>\beta^*$ or $p=2=\beta^*$. Hence we also have 
$\sum_{i\le n}\hat\xi^2_{ni}\cip 0$ for such a~$p$, proving the first claim.
\bigskip

Part (ii).
Assume that $p\in(0,\beta_*)$ and recall that $M_n<3\sqrt{\log n}$ with 
high probability for all large~$n$. Note that $\beta_*>0$ implies $\Pi(\R)=\infty$, 
and so $Y$ is not compound Poisson. Let $I$ be the set of indices $i$ such that 
\begin{equation}
\label{eq:N}\sqrt n |\Delta_i^nY|/\sigma>6\sqrt{\log n}.
\end{equation} 
The cardinality $N=|I|$ is Binomial$(n,p_n)$ distributed, where $p_n$ satisfies  
\[
np_n=n\p\Big(|Y_{1/n}|>6a_n/\sigma\Big)
\geq \frac{1}{2}\ov \Pi(ca_n), 
	\qquad \text{for }a_n=\sqrt{\frac{\log n}{n}},
\]
some $c>0$ and all large $n$, see Lemma~\ref{lem:levy}(a).
 This implies that $np_n\to\infty$ and so 
\[
N=np_n(1+o_\p(1)).
\]
Moreover, Lemma~\ref{lem:levy}(a) shows that $p_n\to 0$ and so $N/n\cip 0$. 

Let $N'$ be the analogue of $N$, but with $6$ replaced by $3$ in~\eqref{eq:N}.
From the definition of $\nu$ (see also~\eqref{eq:order}) and the above bound on 
$M_n$, we conclude that all $Z_{\nu(i)}$, $i\in I$, must be among the $N'$ largest 
or among the $N'$ smallest values of $Z_\cdot$ with high probability. 
As with $N$, we see that $N'/n\cip 0$ and thus 
$Z^{(N')}\cip -\infty$ and $Z^{(n-N')}\cip+\infty$. 
The corresponding $Z_i$, $i\in I$, however, are chosen independently of~$Y$ 
so by the law of large numbers, $\lceil N/2\rceil$ of their moduli $|Z_i|$ must be 
bounded above by $\Phi^{-1}(4/5)$ with high probability for all large~$n$. 
Finally, we get the following bound with high probability for all sufficiently 
large~$n$: 
\[
\sum_{i\le n}(Z_i-Z_{\nu(i)})^2\geq 3N\geq \ov \Pi(ca_n).
\]

Choose $q\in(p,\beta_*)$ and note that necessarily $x^q\ov \Pi(x)\to\infty$ as 
$x\downarrow 0$. Thus for some $c_1>0$ and all large $n$ we have the bound
\[
n^{-p/2}\ov\Pi(ca_n)
\geq c_1n^{-p/2}n^{q/2}(\log n)^{-q/2}\to\infty.
\]
This shows that
\[
\sum_{i\le n}\hat \xi_{ni}^2
=n^{-p/2}\sum_{i\le n}(Z_i-Z_{\nu(i)})^2\cip\infty,
\]
instead of convergence to~0. The proof is now complete.
\end{proof}

\begin{lemma}\label{lem:exch}
Let $(Z_1,\ldots, Z_n)$ be exchangeable and independent of $(Z'_1,\ldots,Z'_n)$. 
For any $1\leq i_1<\cdots<i_k\leq n$ and $y_1,\ldots,y_k\in\R$ define 
\[
\tilde Z_i=Z_i+\sum_{u=1}^k y_{u}\1{i=i_u},\qquad i=1,\ldots,n.
\]
Assume there are no ties a.s.\ and let $\nu$ and $\pi$ be permutations such that 
the orderings of $(\tilde Z_i)$, $(Z_{\nu(i)})$ and $(Z'_{\pi(i)})$ coincide. 
Then the sequence $((Z_i,Z_{\nu(i)},Z'_{\pi(i)}))_{i\notin\{i_1,\ldots,i_k\}}$ of length 
$n-k$ is exchangeable.
\end{lemma}	

\begin{proof}
Any finite sequence of random variables is exchangeable if and only if it can 
be represented as arbitrary random variables that are independently and 
uniformly permuted, see e.g.~\cite[Prop~1.8]{MR2161313}. 
Therefore, we may prove the result by conditioning in a way that only the 
order of the variables $(Z_i)$ is random and then removing such conditioning. 
Thus, we henceforth assume that the sequence $(Z_i')$ is non-random 
and $(Z_i)$ is the result of uniformly permuting non-random numbers. 

Let the permutation $s$ be such that $\tilde Z_{s(1)}<\cdots<\tilde Z_{s(n)}$. 
The permutation $s^{-1}$ maps $s(i_u)$ to $i_u$ for $u=1,\ldots,k$ and is 
otherwise independently and uniformly distributed. 
The sequences $(Z_{\nu(i)})$ and $(Z'_{\pi(i)})$ are obtained by sorting 
$(Z_i)$ and $(Z'_i)$ in increasing order and then permuting according 
to~$s^{-1}$. We conclude that the law of the sequence 
$((Z_i,Z_{\nu(i)},Z'_{\pi(i)}))_{i\notin\{i_1,\ldots,i_k\}}$ is invariant under 
uniform permutations of $\{1,\ldots,n\}\setminus\{i_1,\ldots,i_k\}$, 
completing the proof.
\end{proof}	

\begin{proof}[Proof of Proposition~\ref{prop:CPP}]
As in the Proof of Theorem~\ref{thm:gral_rate}, we consider the increments 
$\xi_{ni}$ but with the scaling $\sqrt{n/(2\log n)}$.
Let $Z^{(i)}$ be the corresponding order statistics and recall that 
(see~\cite[Thm~1]{MR808861}), as $n\to\infty$, 
\[
Z^{(n)}-\sqrt{2\log n}\cip 0
\qquad\text{ and }\qquad
\Delta_n:=\max_{i<n}(Z^{(i+1)}-Z^{(i)})\cip 0.
\]

First, we focus on the partial sum process corresponding to
\[
\hat\xi_{ni}=(2\log n)^{-1/2}(Z_i-Z_{\nu(i)}).
\]  
We will work on the event $\{N_\pm=k_\pm\}$ for $k=k_++k_-\geq 1$, where 
$N_\pm$ is the number of positive/negative jumps~$J_\cdot$; the case of no jumps 
is trivial. Now the following is true for large $n$ with probability arbitrarily close 
to~1. The indices $\lceil T_j n\rceil$ must be different (the set of such is denoted 
by $I$), every $\sqrt n|J_j|/\sigma$ must be larger than $Z^{(n)}-Z^{(1)}$, and the 
latter is smaller than $3\sqrt{\log n}$. Hence for each $i=\lceil nT_j\rceil\in I$ the 
quantity $Z_i+\sqrt n\Delta_i^nY/\sigma$ must be among $k_+$ largest if the 
corresponding $J_j>0$ or $k_-$ smallest if $J_j<0$. Thus 
\[
(2\log n)^{-1/2}Z_{\nu(i)}\cip \pm 1
\]
according to the sign of the respective jump $J_j$. However, the variables 
$Z_i$ do not depend on the choice of indices~$i$, so
\[
\hat\xi_{ni}\cip-\sign(J_j),\qquad i/n=\lceil T_jn\rceil/n\cip T_j,
\]
where $j$ is the corresponding jump index. It is thus left to show that
the partial sum process of $\hat\xi_{ni}$ with $i\in I$ excluded converges in probability to $(k_+-k_-)t$ in supremum norm.
But the vector $\hat\xi_{ni},i\notin I$ is also exchangeable, see Lemma~\ref{lem:exch}, and so according to~\cite[Thm~3.13]{MR2161313} it is sufficient to show that
\begin{equation}\label{eq:sumsto0}
\sum_{i\notin I}\hat\xi_{ni}\cip k_+-k_-
\qquad \text{and} \qquad\sum_{i\notin I}\hat\xi^2_{ni}\cip 0.
\end{equation}

Since we only need to look at the sums, we may permute the indices arbitrarily. 
In this paragraph we assume that $Z_\cdot$ is an increasing sequence, and that the 
elements of $I$ are given by $i_1<\cdots<i_k$. For $i>i_k$ we have $\nu(i)=i-k_+$, 
for $i<i_1$ we have $\nu(i)=i+k_-$ and between any two $i_j$ and $i_{j+1}$, 
the permutation $\nu$ displaces every index 
a fixed amount bounded by $k$. Furthermore, the indices $i_j$ are chosen uniformly 
at random (and then sorted), implying $(Z_{i_k}-Z_{i_1})/\sqrt{2\log n}\cip 0$. 
Thus 
\[
\sum_{i\notin I,\, i_1<i<i_k}|\hat\xi_{ni}|
\le k\frac{Z_{i_k}-Z_{i_1}}{\sqrt{2\log n}}\cip 0,
\]
\[
\sum_{i<i_1}\hat\xi_{ni}
= \sum_{j\le k_-}\frac{Z_j-Z_{i_1+j-1}}{\sqrt{2\log n}}\cip -k_-
\enskip\text{and}\enskip
\sum_{i>i_k}\hat\xi_{ni}
= \sum_{j\le k_+}\frac{Z_{n-j+1}-Z_{i_k-j+1}}{\sqrt{2\log n}}\cip k_+,
\]
which yield the first limit in~\eqref{eq:sumsto0}. A simple induction on $k$ 
shows that the bound $\sum_{i\notin I}|Z_i-Z_{\nu(i)}|\le k(Z_n-Z_1)$ holds, 
establishing the second limit in~\eqref{eq:sumsto0}:
\[
\sum_{i\notin I}\hat\xi^2_{ni}\le 
	\frac{\Delta_n}{2\log n}\sum_{i\notin I}|Z_i-Z_{\nu(i)}|
\le k\Delta_n \frac{Z_n-Z_1}{{2\log n}} \cip 0.
\]	

It remains to show that the partial sums of $\tilde \xi_{ni}$ vanish in probability. 
Observe that the sequence $\tilde \xi_{ni}$ need not be exchangeable. Nevertheless, 
we may condition on the number of jumps and note that $(\tilde \xi_{ni})_{i\notin I}$ 
(of length $n-k$) is exchangeable. Indeed, we need only apply Lemma~\ref{lem:exch} 
after conditioning on the ordered values of $Z$ and $Z'$.
Now $\sum_{i\notin I}\tilde\xi_{ni}^2\leq \sum_{i\le n}\tilde\xi_{ni}^2\cip 0$ according to Lemma~\ref{lem:normal}. Moreover,
\[
\sum_{i\notin I}\tilde\xi_{ni}=-\sum_{i\in I}\tilde\xi_{ni}\cip 0,
\]
because for $i\in I$, both $Z_{\nu(i)}$ and $Z'_{\pi(i)}$ become 
$\pm\sqrt{2\log n}+o_\p(1)$ (with the same sign) and hence $\tilde \xi_{ni}\cip 0$. 
This yields 
\[
\sum_{i\notin I}\tilde\xi_{ni}\cip 0,\qquad \sum_{i\in I}\tilde\xi_{ni}\cip 0,
\]
completing the proof.
\end{proof}

\begin{proof}[Proof of Proposition~\ref{prop:nosigma}]
Note that the bivariate increments $\xi_{ni}=(\Delta_i^n X,\Delta_i^n W^\n)$ are 
exchangeable. Moreover, the partial sums of the first coordinate corresponds to the 
process $X$ observed on the grid $1/n,\ldots,1$, and those of the second coordinate 
correspond to some Brownian motion (dependent on $X$) observed on the same grid. 
Now we apply~\cite[Thm~3.13]{MR2161313} to each coordinate separately, 
and then jointly. It is only required to show that the cross-variation vanishes:
\[
\sum_{i\leq n}(\Delta_i^n X)(\Delta_i^n W^\n)\cip 0.
\]

Recall that $\max_{i\le n} |\Delta_i^n W'|=\Oh_\p(\sqrt{\log n/n})$, and hence we 
are done in the case when $X$ has bounded variation on compacts. In general, 
by~\cite[Thm~2.3]{jacod_asymptotic}, it is sufficient to show that 
\[
\sqrt {n\log n}\e(|X_{1/n}|\wedge 1)\to 0,
\]
 so Lemma~\ref{lem:levy}(c) completes the proof. 
\end{proof}

\section{Some estimates for L\'evy processes}\label{sec:levy}
This section is devoted to some basic bounds for L\'evy processes at small times. 
Here we prove the three statements in Lemma~\ref{lem:levy}, and also lay foundations needed in the proofs underlying Method~II in \S\ref{sec:proofsII}.
Recall that $(\gamma, 0,\Pi)$ is the L\'evy triplet of $Y$ having no Brownian part. 
For any $x\in(0,1]$ define the standard quantities: 
\[
m(x)=\gamma-\int_{x\leq |y|<1}y\Pi(\D y),\qquad
v(x)=\int_{|y|< x}y^2\Pi(\D y).
\]
In the case when $Y$ has bounded variation on compacts we can express the linear drift as $\gamma_0=m(0)$.
We also let 
\[
Y_t=m(x)t+J_t^{x,1}+J_t^{x,2}
\] 
be the \levy-It\^o decomposition of $Y$, 
where $J_t^{x,1}$ is the martingale corresponding to the compensated jumps 
of $Y$ of magnitude less than $x$ and $J_t^{x,2}$ is driftless compound 
Poisson process containing all jumps of $Y$ of magnitude at least~$x$. 
In particular, $\e[(J_t^{x,1})^2]=v(x)t$. Finally, we consider the integrals 
\[
I_q=\int_{(-1,1)}|x|^q\Pi(\D x),\qquad q\ge0.
\] and recall the following  useful lemma 
(see, e.g.~\cite[Lem.~9]{LevySupSim}).
\begin{lemma}\label{lem:BG-bounds}
If $I_q<\infty$ for some $q\in[0,2]$, then for any $x\in(0,1]$, we have 
\[
\ov\Pi(x)\le\ov\Pi(1)+I_q x^{-q},\quad
|m(x)|\le |\gamma|+I_q x^{-(q-1)^+},\quad
v(x)\le I_q x^{2-q}.
\]
\end{lemma}
Note that we may always choose $q=2$, and even $q=1$ when $Y$ is of bounded 
variation on compacts.
Next, we establish some estimates on the truncated moments.
\begin{lemma}\label{lem:trunc_mom}
For any $p\in(0,2]$, $K>0$, $t>0$ and $x\in(0,1)$, we have
\begin{align*}
\e(|Y_t|^p\wedge K)
&\le(m(x)^2 t^2 + v(x) t)^{p/2}+ K\ov\Pi(x)t,
\\
\p(|Y_t|\geq K)
&\le (m(x)^2 t^2 + v(x) t)/K^{2}+ \ov\Pi(x)t.
\end{align*}
\end{lemma}
\begin{proof}
Fix $t>0$ and define the event $A=\bigcap_{s\le t}\{J_s^{x,2}=0\}$ of not 
observing any jump from $J_s^{x,2}$ on the time interval $[0,t]$. 
Clearly $1-\p(A)=1-e^{-\ov\Pi(x)t}\le \ov\Pi(x)t$. 
Consider the elementary inequality
$|Y_t|^p\wedge K\le |m(x)t+J_t^{x,1}|^p1_{A}+K1_{A^c}$. Taking expectations 
and applying Jensen's inequality we obtain the bound
\[
\e\big(|Y_t|^p\wedge K\big)
\le \big(m(x)^2t^2 + \e\big[(J_t^{x,1})^2\big]\big)^{p/2}
+ K(1-\p(A)),
\]
because $\e J_t^{x,1}=0$. The first inequality readily follows. Using Markov's inequality we readily get
\[
\p(|Y_t|\geq  K)
=\p(|Y_t|\wedge K\geq  K)\leq \e(Y_t^2\wedge K^2)/K^2
\]
and the second result follows from the first with $p=2$. 
\end{proof}

\begin{lemma}\label{lem:lower_boundY}
For any $\epsilon>0$ and $a_t\downarrow 0$ satisfying $a_t/\sqrt t\to \infty$ 
as $t\downarrow 0$ we have
\[ 
\liminf_{t\downarrow 0}\frac{\p(|X_t|>a_t)}{t\ov\Pi(a_t(1+\epsilon))}
\ge 1.
\]
\end{lemma}	

\begin{proof}
Take $x=x_t=a_t(1+\epsilon)$ and consider the event that $J^{x,2}_s$ has exactly 
one jump in $[0,t]$, which yields the lower bound 
\[
\p(|X_t|>a_t)
\ge t(1+o(1))\ov\Pi(x_t)\p\big(|\tilde J^{x_t,1}_t|+|m(x_t)|t<a_t\epsilon\big),
\]
where $\tilde J^{x,1}=J^{x,1}+\sigma W$. Here we use $t\ov\Pi(x_t)\to 0$ which 
follows from Lemma~\ref{lem:BG-bounds} with $q=2$ and the assumption 
$ta_t^{-2}\to 0$. Furthermore, we have $|m(x_t)|t/a_t\to 0$ and 
$\p(|\tilde J^{x_t,1}_t|>a_t\epsilon/2)\to 0$ which follows from Markov's 
inequality and the fact that 
$\e[(\tilde J^{x_t,1}_t)^2]/a_t^2=t(\sigma^2+v(x_t))/a_t^2\to 0$. 
This completes the proof.
\end{proof}

\begin{proof}[Proof of Lemma~\ref{lem:levy}]
Part (a). 
The inequality follows from Lemma~\ref{lem:lower_boundY}. The limit is a 
consequence of the second inequality in Lemma~\ref{lem:trunc_mom} with 
$t=1/n$ and $K=x=a_n$, Lemma~\ref{lem:BG-bounds} with $q=2$ and the fact 
that $a_n\sqrt{n}\to\infty$. 

Part (b). 
From Lemma~\ref{lem:trunc_mom} with $x_n^2=n^{-1}\log n$ we have the bound
\[
n^{2-p/2}\e\Big(Y_{1/n}^2\wedge \frac{\log n}{n}\Big)
\le n^{-p/2}m(x_n)^2 + n^{1-p/2}v(x_n) + n^{-p/2}\log(n) \ov\Pi(x_n).
\]
Assume that $\beta^*<2$, pick $q<p$ such that $I_q<\infty$, and apply 
Lemma~\ref{lem:BG-bounds}. The first term vanishes because $-p/2+(q-1)^+\leq 0$. 
The second term vanishes because $1-p/2-(2-q)/2<0$.
The third term vanishes since $-p/2+q/2<0$. Finally, if $\beta^*=2$ then 
taking $q=p=2$, proceeding as in the previous case and using the facts that 
$x^2\ov\Pi(x)\to0$ and $v(x)\to0$ as $x\to0$, gives the result.

Part (c). 
Applying Lemma~\ref{lem:trunc_mom} with $x=n^{-1/4}$ to $Y$ gives 
\begin{align*}
n\log(n)[\e(|Y_{1/n}|\wedge 1)]^2 
\le&\, 2n^{-1}\log(n)m\big(n^{-1/4}\big)^2
	+2\log(n) \big(v\big(n^{-1/4}\big)-\sigma^2\big)\\ 
	&+ 2n^{-1}\log(n)\ov\Pi\big(n^{-1/4}\big)^2.
\end{align*}
Take $q$ satisfying $\beta^*\vee 1<q<2$ and apply 
Lemma~\ref{lem:BG-bounds} to show that this quantity indeed tends to~0.
\end{proof}

\section{Proofs for Method II}
\label{sec:proofsII}
Without loss of generality we assume throughout this section that~$\sigma=1$, 
which is based on a simple rescaling argument. The main ingredient in the proof 
of Theorem~\ref{thm:tunc_method} is the following result.
\begin{lemma}\label{lem:trunc1}
Assume that $\int_{|x|>1} x^2\Pi(\D x)<\infty$. Then
for $p\in(\beta^*,2]$ we have
\begin{align*}
&n^{2-p/2}\e \big(Y^2_{1/n}\1{|X_{1/n}|\leq a_n}\big)\to 0,
 &n^{2-p/2}\e \big(W^2_{1/n}\1{|X_{1/n}|>a_n}\big)\to 0,
\end{align*}
assuming that the positive sequence $a_n$ satisfies
\[
\liminf_{n\to\infty}\frac{na^2_n}{\log n}>2-p
	\qquad\text{and}\qquad
a_nn^{1/2-\delta}\to 0\quad\text{for some}\quad 
	\delta<\frac{p-\beta^*}{2(2-\beta^*)}.
\]
This result is also true for $p=2$ and any $\beta^*$, in which case the upper bound on 
$a_n$ is replaced by $a_n\to 0$.
\end{lemma}
\begin{proof}
The set inclusion $\{|X_{1/n}|\leq a_n\}
\subset\{|Y_{1/n}|\leq 2a_n\}\cup\{|W_{1/n}|>a_n\}$
implies 
\[
\e\big(Y_{1/n}^2\1{|X_{1/n}|\leq a_n}\big)
	\le\e\big(Y_{1/n}^2\1{|Y_{1/n}|\leq 2a_n}\big)
		+ \e\big(Y_{1/n}^2\1{|W_{1/n}|> a_n}\big).
\]
Since $\e(Y_{1/n}^2)\sim \int_\R x^2\Pi(\D x)/n$, it follows by Mill's ratio that 
\begin{equation}
\label{eq:trunc_Y_1W}
\e\big(Y_{1/n}^2\1{|W_{1/n}|> a_n}\big)
	=2\e(Y_{1/n}^2)\Phi(-\sqrt{n}a_n)
	\sim\frac{2e^{-na^2_n/2}}{\sqrt{2\pi}n^{3/2}a_n}\int_\R x^2\Pi(\D x).
\end{equation} 
The assumed lower bound on $a_n$ implies that $a_n^2n/2>(1-p/2)\log n$ 
for all large enough $n$, showing that the term in~\eqref{eq:trunc_Y_1W} is 
$o(n^{p/2-2})$; the case $p=2$ needs special attention but is otherwise 
straightforward. To bound the other expectation, we use 
Lemmas~\ref{lem:trunc_mom} and~\ref{lem:BG-bounds} to obtain 
\begin{align}
\nonumber
\e\big(Y_{1/n}^2\1{|Y_{1/n}|\leq 2a_n}\big)
&\le m(a_n)^2/n^2+v(a_n)/n+4a^2_n\ov\Pi(a_n)/n\\
\label{eq:trunc_Y_1Y}
&\le 2\big(\gamma^2+I_q^2a_n^{-2(q-1)^+}\big)/n^2
	+5 I_q a_n^{2-q}/n+4\ov\Pi(1)a_n^2/n,
\end{align}
for any $q\in(\beta^*,p)$ whenever $a_n\le 1$ (i.e., for all large enough $n$). 
Note that $na_n^q\to\infty$ and so it is enough to establish an upper bound on 
$a_n^{2-q}/n$. But we have assumed that this term is 
$n^{-1+(2-q)(-1/2+\delta)}o(1)$ for certain $\delta$.
Equating the power to $-2+p/2$ we find that 
$\delta=(p-q)/[2(2-q)]\uparrow (p-\beta^*)/[2(2-\beta^*)]$ as 
$q\downarrow \beta^*$, and hence we have the bound $o(n^{-2+p/2})$. 
If $p=2$ then we may take $q=2$ to see that 
the result is $o(n^{-1})$ as claimed.

With regard to the second statement, we fix any sequence $b_n\downarrow0$ 
with values in $(0,1)$ and denote $c_n=a_n(1-b_n)$. As before, 
we have the inequality 
\[
\e\big(W_{1/n}^2\1{|X_{1/n}|>a_n}\big)
\le \e\big(W_{1/n}^2\1{|Y_{1/n}|>a_nb_n}\big)
	+\e\big(W_{1/n}^2\1{|W_{1/n}|>c_n}\big).
\]
The second term on the right may be written as 
\begin{equation}
\label{eq:trunc_W_1W}
2n^{-1}\e\big(W_1^2\1{W_1>\sqrt nc_n}\big)
=2n^{-1}\int_{c_n\sqrt{n}}^\infty x^2\frac{e^{-x^2/2}}{\sqrt{2\pi}}\D x
\sim\frac{2c_n e^{-nc^2_n/2}}{\sqrt{2\pi n}}.
\end{equation}
This term is only made larger by taking $c_n$ smaller, and so we may assume 
that $c_n^2n/2=(1-p/2+\epsilon)\log n$ for some $\epsilon>0$ and all large $n$. 
Thus the term in~\eqref{eq:trunc_W_1W} is $o(n^{p/2-2})$. On the other hand, 
according to Lemma~\ref{lem:trunc_mom}, the first term satisfies 
\begin{align}
\nonumber
\e\big(W_{1/n}^2\1{|Y_{1/n}|>a_nb_n}\big)
&= \p(|Y_{1/n}|>a_nb_n)/n\\
\nonumber
&\le (m(a_n)^2 /n^3 + v(a_n) /n^2)a_n^{-2}b_n^{-2}+ \ov\Pi(a_n)/n^2,
\end{align}
when $a_n<1$.
Since $na_n^2\to \infty$, the argument used in~\eqref{eq:trunc_Y_1Y} completes the proof upon taking 
$b_n$ such that $na_n^2b_n^2\geq 1$.
\end{proof}

\subsection{Convergence results}
\begin{proof}[Proof of Theorem~\ref{thm:tunc_method}]
Without loss of generality we assume that the jumps of $X$ are 
bounded, since they are below a threshold $K\to\infty$ with probability tending to 1. As in the 
proof of Theorem~\ref{thm:gral_rate}, it suffices to prove the result when 
the supremum is taken over the grid $(i/n)_{i=1,\ldots,n}$. 
On that grid, $\hat W^\n-W$ is a random walk and so 
$(W_{i/n}-\hat W_{i/n}^\n-(W_1-\hat W_1^\n)i/n)_{i=1,\ldots,n}$ has 
exchangeable increments, while being~0 at~$i=n$. Thus, the first assertion is equivalent to: 
\[
n^{(2-p)/2}\sum_{i\le n} \Big(V^\n_i-\frac{1}{n}\sum_{j\le n}V^\n_j\Big)^2\cip 0,
\quad\text{where}\quad
V^\n_i
=\begin{cases}\Delta ^n_i W, &|\Delta ^n_i X|>a_n,\\
	-\Delta ^n_i Y, &|\Delta ^n_i X|\leq a_n.
\end{cases}
\]
This will be shown by proving convergence in $L^1$. As before, the expectation of the sum above may be rewritten as 
\begin{equation}
\label{eq:compensator}
(n-1)\left(\e\big[\big(V^\n_1\big)^2\big]-\big(\e V^\n_1\big)^2\right).
\end{equation}
But now the result follows from 
\begin{equation*}
n^{2-p/2}\e \big(V^\n_1\big)^2
=n^{2-p/2}\e \big(W^2_{1/n}\1{|X_{1/n}|>a_n}\big)
	+n^{2-p/2}\e \big(Y^2_{1/n}\1{|X_{1/n}|\leq a_n}\big)\to 0,
\end{equation*}
where in the last step we apply Lemma~\ref{lem:trunc1}.
\end{proof}

\begin{proof}[Proof of Proposition~\ref{prop:drift}]
In view of Theorem~\ref{thm:tunc_method} it is sufficient to show the stated convergence for $t=1$, which is equivalent to 
\[
n^{(1-p)/2}\Big(\sum_{i\le n}V^\n_i+\gamma_0\Big)\cip 0.
\]
Arguments similar to those in Lemma~\ref{lem:trunc1} give 
$n^{3/2-p/2}\e\big(|W_{1/n}|\1{|X_{1/n}|> a_n}\big)\to 0$. 
Hence we only need to show that
\[
n^{(1-p)/2}\Big(\gamma_0-\sum_{i\leq n} \Delta_i^n Y\1{|\Delta_i^n X|
	\leq a_n}\Big)\cip 0,
\]
and for this the following is sufficient:
\[
n^{(1-p)/2}\e\Big|\gamma_0-nY_{1/n}\1{|X_{1/n}|\leq a_n}\Big|\to 0.
\]
Write $Y_{1/n}=\gamma_0/n+J_{1/n}$, where $J_t$ is the uncompensated sum of 
the jumps of $Y$ on $(0,t]$. We note that $n^{(1-p)/2}\p(|X_{1/n}|>a_n)\to 0$ 
so it remains to prove 
\[
n^{3/2-p/2}\e(|J_{1/n}|\1{|X_{1/n}|\leq a_n})\to 0.
\]
Again, by the arguments in Lemma~\ref{lem:trunc1} we may replace the indicator 
by $\1{|J_{1/n}|\leq 2a_n}$. It is left to note that the simple structure of $J$ allows for an improved bound as compared to Lemma~\ref{lem:trunc_mom}:
\[
\e (|J_{1/n}|\wedge a_n)
\le n^{-1}\int_{(-a_n,a_n)} |y|\Pi(\D y) + a_n\overline\Pi(a_n)/n,
\]
where the first part corresponds to the sum of the absolute jumps of $J$ of size 
smaller than $a_n$ and the second part to $a_n$ times the probability of 
observing at least one jump whose size is at least~$a_n$ in absolute value. 
Both terms are $o(a_n^{1-q}/n)$ for $q>\beta^*$ and hence $o(n^{-3/2+p/2})$, 
proving the main claim. For $p=1$ we additionally observe that the display above 
is $o(n^{-1})$ because necessarily $x\overline\Pi(x)\to 0$ as $Y$ is of bounded 
variation. 
\end{proof}

\subsection{Cases when convergence fails}

Here we prove the negative results yielding Proposition~\ref{prop:neg}.
Consider the $n/2$ epoch, assuming $n$ is even in the following, and rewrite 
the difference of interest as the sum of independent symmetric terms:
\[
n^{(2-p)/4}(W_{1/2}-\hat W_{1/2}^\n-(W_1-\hat W_1^\n)/2)
=\frac{1}{2}\sum_{i\leq n/2}n^{(2-p)/4}(V_i^\n-V^\n_{n/2+i}).
\] 
Now the standard result~\cite[Ex.~4.18]{MR1876169} states that converges to~0 
in probability is equivalent to convergence to~0 of the sum of the expected 
(truncated) squares:
\begin{equation}\label{eq:disprove}
n\e\Big(n^{1-p/2}(V_1^\n-V^\n_{2})^2\wedge 1\Big)\to 0.
\end{equation}
Hence we only need to disprove the latter.

First, we consider the case when the threshold $a_n$ is too small:
\begin{lemma}\label{lem:neg1}
If $\liminf na^2_n/\log n<2-\beta^*$ then~\eqref{eq:disprove} 
fails for some $p\in(\beta^*,2]$.
\end{lemma}

\begin{proof}
We may assume that $a_n>n^{-1/2}$ for all large enough~$n$, because otherwise 
$V_1^\n=W_{1/n}$ with probability bounded away from~0, and the contradiction 
can be easily derived. To get an appropriate lower bound on the expectation 
in~\eqref{eq:disprove}, we consider the event 
\[
\{\Delta_1^n W>a_n,\Delta_1^n Y\geq 0,
	|\Delta_2^n W|\leq a_n/2,|\Delta_2^n Y|\leq a_n/2\}.
\]
We may assume that $\p(Y_{1/n}\geq 0)\geq 1/2$ since we may flip the signs 
otherwise. On this event we have $V_1-V_2\geq a_n/2$. Using the independence 
and the fact that $\p(|\Delta_2^n W|\leq a_n/2,|\Delta_2^n Y|\leq a_n/2)$ is 
bounded away from~0 (as $a_n>n^{-1/2}$ for large $n$) we get a bound
\begin{align*}
n\e\Big(n^{1-p/2}(V_1^\n-V^\n_{2})^2\wedge 1\Big)
&\ge c n(a_n^2n^{1-p/2}\wedge 1)\p(W_{1/n}> a_n)\\
&\ge c n^{1-p/2}\p(W_{1/n}> a_n)
\end{align*}
for some $c>0$. 
Note that $na^2_n/(2\log n)<1-\beta^*/2-\epsilon$ for some $\epsilon>0$ along 
some subsequence. Along that subsequence we have 
$\p(W_1> a_n\sqrt n)\geq n^{-1+\beta^*/2+\epsilon/2}$ for large~$n$. Thus the 
contradiction is obtained for any $p<\beta^*+\epsilon/2$. 
\end{proof}

Finally, we consider the case when $a_n$ is too large:
\begin{lemma}\label{lem:neg2}
Assume that $a_n\to 0$ and $\Pi\ne 0$. 
If $\limsup_{n\to\infty} a_nn^{1/2-\epsilon}>0$ for some $\epsilon>0$ 
then~\eqref{eq:disprove} fails for some $p\in(\beta_*,2]$.
\end{lemma}

\begin{proof}
First suppose $\beta_*>0$. 
Without real loss of generality we may assume that $a_n>n^{-1/2+\epsilon}$ 
for some $\epsilon\leq \beta_*/4$.
In this case we focus on the event
\[
\{|\Delta_1^n W|\leq a_n/2,a_n/4< |\Delta_1^n Y|\leq a_n/2,
	|\Delta_2^n W|\leq a_n/2,|\Delta_2^n Y|\leq a_n/8\},
\]
yielding a lower bound
\begin{equation}
\label{eq:neg2-low}
n\e\Big(n^{1-p/2}(V_1^\n-V^\n_{2})^2\wedge 1\Big)\geq cn^{2-p/2}a_n^2\p(|Y_{1/n}|>a_n/4)
\end{equation}
for large enough~$n$. 
According to Lemma~\ref{lem:lower_boundY} this is further lower bounded by
\[
c'n^{1-p/2}a_n^2\overline\Pi(a_n/2)
\geq c''n^{1-p/2}a_n^{2-\nu}
\geq c''n^{1-p/2}n^{(2-\nu)(-1/2+\epsilon)},
\]
for some $c',c''>0$ and $\nu<\beta_*$. 
The power is non-negative when $p\leq \nu+2\epsilon(2-\nu)$, but the right 
side can be made larger than $\beta_*$ by taking $\nu$ close to $\beta_*$; 
recall that our assumptions imply $\beta_*<2$.
In the case $\beta_*=0$ we choose $p<4\epsilon$ and lower bound the 
expression in~\eqref{eq:neg2-low} by some positive number as before, 
since $\Pi(\R)>0$.
\end{proof}

\begin{proof}[Proof of Proposition~\ref{prop:neg}]
Lemma~\ref{lem:neg1} and Lemma~\ref{lem:neg2} establish the second and third 
claims, respectively. The proof of first claim is analogous to that in 
Lemma~\ref{lem:neg2}. Note that the assumptions imply both $a_n\to 0$ and 
$\beta_*>0$, so $\Pi\ne0$. 
\end{proof}

\begin{appendix}
\section{Purely Brownian case}\label{sec:normal}
An interesting problem is to identify the exact rate of convergence of the 
skeletons in Proposition~\ref{prop:CPP} for the purely Brownian case, that is, 
when $Y=0$. Here we show that this rate is $\sqrt{\log\log n/n}$. We cannot, 
however, establish the limit law, nor its existence.

As in \S\ref{sec:proofs}, consider the variables 
\[
\xi_{ni}=\frac{1}{\sqrt{\log\log n}}
\bigg(Z_{\nu(i)}-Z'_{\pi(i)}-\frac{1}{n}\sum_{j\le n}(Z_j-Z'_j)\bigg),
\qquad i=1,\ldots,n.
\]
For any $n\in\N$, let $S^\n_t=\sum_{i\leq tn}\xi_{ni}$ be its cumulative sum process. We clearly have $S^\n_0=S^\n_1=0$ and the jumps 
of $S^\n$ are exchangeable. 

Following the proof of Lemma~\ref{lem:normal}, 
more specifically, the bounds in terms of the functions $\tilde{F}_n$ and 
$\tilde{G}_n$, we easily deduce that for any $a>4$, the quadratic variation of 
$S^\n$ satisfies 
\[
\p\bigg(\sum_{i\le n} {\xi}_{ni}^2>a\bigg)\to 0.
\] 
The stated tightness in turn implies, according 
to~\cite[Lem.~3.9]{MR2161313}, that the processes $S^\n$ are tight 
(and nonvanishing) in the Skorohod space $\mathcal{D}[0,1]$.
This establishes the claimed rate of convergence of skeletons.
\end{appendix}

\section*{Acknowledgments}
JGC is grateful for the support of The Alan Turing Institute under  
EPSRC grant EP/N510129/1 and CoNaCyT scholarship 
2018-000009-01EXTF-00624 CVU699336.
JI gratefully acknowledges financial support of Sapere Aude Starting Grant 
8049-00021B ``Distributional Robustness in Assessment of Extreme Risk''.


\end{document}